\documentclass[A4, 8pt]{article}
\usepackage[left=2cm,right=2cm,top=2.54cm,bottom=2.54cm]{geometry}
\usepackage[utf8]{inputenc}
\usepackage{array}
\usepackage{url}
\usepackage{caption}
\usepackage{subcaption}
\usepackage{bbm}
\usepackage{enumitem}   
\usepackage{mathtools}
\usepackage{graphicx}

\newcommand{\bq}{\begin{eqnarray*}}
\newcommand{\bqn}[1]{\begin{eqnarray}\label{#1}}
\newcommand{\eq}{\end{eqnarray*}}
\newcommand{\eqn}{\end{eqnarray}}

\usepackage{amsthm, amsmath, amsfonts, mathtools, amssymb} 
\DeclareMathOperator*{\argmax}{argmax}

\usepackage{sgame, tikz} 
\usetikzlibrary{trees, calc} 
\usepackage[countmax]{subfloat}
\usepackage{hyperref} 
\usepackage{cases}
\usepackage{float}

\usepackage{listings}
\usepackage{csvsimple}

\usepackage{xcolor}

\definecolor{customred}{rgb}{0.6, 0, 0}

\definecolor{codegray}{rgb}{0.5,0.5,0.5}
\definecolor{codepurple}{rgb}{0.54, 0.17, 0.89}
\definecolor{codegreen}{rgb}{0.55, 0.71, 0.0}
\definecolor{backcolour}{rgb}{1, 1, 1}
\definecolor{blue(pigment)}{rgb}{0.2, 0.2, 0.6}
\definecolor{ao(english)}{rgb}{0.0, 0.5, 0.0}

\usepackage{scalerel}

\lstdefinestyle{mystyle}{
    backgroundcolor=\color{backcolour},   
    commentstyle=\color{codepurple},
    keywordstyle=\color{codegray},
    numberstyle=\tiny\color{codegray},
    stringstyle=\color{codegreen},
    basicstyle=\ttfamily\footnotesize,
    breakatwhitespace=false,         
    breaklines=true,                 
    captionpos=b,                    
    keepspaces=true,                 
    numbers=left,                    
    numbersep=5pt,                  
    showspaces=false,                
    showstringspaces=false,
    showtabs=false,                  
    tabsize=2
}

\numberwithin{equation}{section}

\newtheorem{theorem}{Theorem}[section]

\newtheorem{proposition}[theorem]{Proposition}

\newtheorem{lemma}[theorem]{Lemma}
\newtheorem{cor}[theorem]{Corollary}
\newtheorem{rem}[theorem]{Remark}

\newcommand\numberthis{\addtocounter{equation}{1}\tag{\theequation}}
\newcommand{\R}{\mathbb{R}}

\newcommand{\E}{\mathbb E}
\newcommand{\Q}{\mathcal Q}

\newtheoremstyle{break}
  {\topsep}{\topsep}%
  {\itshape}{}%
  {\bfseries}{}%
  {\newline}{}%
\theoremstyle{break}

\allowdisplaybreaks

\begin{document}
\thispagestyle{empty}
	\begin{center}
		\Large{\textbf{A forward algorithm for a class of Markov zero-sum stopping games}} \\
        \begingroup
            \centering
            \normalsize{Nhat-Thang Le} \\
            Toulouse School of Economics \\
            Institut de Mathématiques de Toulouse \\
            University of Toulouse \\[2mm]
            
        \endgroup
        \end{center}
\setbox5=\vbox{
\hbox{lenhat.thang@tse-fr.eu\\[1mm]}
\vskip1mm
\hbox{Toulouse School of Economics,\\}
\hbox{1, Esplanade de l'université,\\}
\hbox{31080 Toulouse cedex 6, France.\\}
\hbox{Institut de Mathématiques de Toulouse,\\}
\hbox{Université Paul Sabatier, 118, route de Narbonne,\\}
\hbox{31062 Toulouse cedex 9, France.\\[1mm]}
}
        \begin{abstract}
            In this paper, we propose a new efficient algorithm to compute the value function for zero-sum stopping games featuring two players with opposing interests. This can be seen as a game version of the ``forward algorithm" for (one-player) optimal stopping problem, first introduced by Irle \cite{Irle} for discrete-time Markov chains and later revisited by Miclo \& Villeneuve \cite{MS} for continuous-time Markov processes on general state spaces. This paper focuses on a game driven by a homogeneous continuous-time Markov chain taking values in a finite state space and also discusses about the number of iterations needed. Illustrated computational implementations for a few particular examples are also provided.
        \end{abstract}
{\small
\par
\vskip.3cm
\textbf{Fundings: }
This work  was supported by the grants CIMI, MINT.
}

\pagenumbering{arabic}

\section{Introduction}
 In this paper, we consider a discounted zero-sum stopping game featuring a sup-player and an inf-player. The sup-player selects a stopping time $\tau$ to maximize, while the inf-player selects a stopping time $\gamma$ to minimize, the expected payoff
    \begin{align} \label{expected payoff}
        R_x(\tau,\gamma) \coloneqq  \mathbb E_x[\mathrm{e}^{-\beta \tau}\psi(X_\tau) \mathbf{1}_{\tau \leq \gamma} + \mathrm{e}^{-\beta \gamma}\phi(X_\gamma) \mathbf{1}_{\tau >\gamma} ],
    \end{align}
where $\beta >0$ is a fixed discounted rate. Here, $(X_t)_{t\geq 0}$ is a right-continuous, homogeneous continuous-time Markov chain taking values in a finite state space $E$, equipped with probability measures $(\mathbb P_x)_{x \in E}$ such that $\mathbb P_x(X_0 = x) = 1$ for all $x \in E$, and admits a Markov generator $\mathcal{Q}\coloneqq  (Q(x,y))_{x,y \in E}$. The payoff functions $\psi,\  \phi \in \R^E$ (the set of all functions on $E$) are given such that $0 \leq \psi \leq \phi$. Since $E$ is finite, we also use the convention that $\mathrm{e}^{-\beta \eta} f(X_\eta) = 0$ on the set $\{ \eta = +\infty\}$ for any stopping time $\eta$ and $f\in \R^E$. Define respectively the upper value and lower value functions by
    \begin{align*}
        \forall x \in E, \qquad \overline V(x) \coloneqq  \inf_\gamma \sup_\tau R_x(\tau, \gamma) \qquad \text{and} \qquad \underline V(x) \coloneqq  \sup_\tau \inf_\gamma  R_x(\tau, \gamma),
    \end{align*}
where the suprema and infima are taken over the set $\mathcal M$ of all stopping times (with respect to the natural filtration of $X$). It is easy to see that 
    \begin{align*}
        \forall x \in E, \quad \psi(x) \leq \underline V(x) \leq \overline V(x) \leq \phi(x).
    \end{align*}
If in addition, we have $ \underline V \geq \overline V $, i.e. $  \underline V = \overline V$, the game is said to have a value. In such cases, we will denote the common value function as $V$. Now, suppose that there are two stopping times $(\tau^*, \gamma^*)$ satisfying 
    \begin{align} \label{intro: NE criteria}
       \forall \tau, \gamma \in \mathcal M, \quad  R_x(\tau, \gamma^*) \leq R_x(\tau^*, \gamma^*) \leq R_x(\tau^*, \gamma) ,
    \end{align}
the pair $(\tau^*,\gamma^*)$ is then referred to as a Nash equilibrium (NE) or a saddle point. Clearly, if there exists such a NE, then the game has a value and the value function is given by $V(x) = R_x (\tau^*,\gamma^*)$ for all $x \in E$.

It is well-known that, when the state space $E$ is finite, the game always has a value and the pair $(\tau',\gamma')$, defined by 
    \begin{align} \label{intro: NE pair}
        \tau' \coloneqq  \inf\{ t\geq0: \ V(X_t) = \psi(X_t) \}\qquad \text{and} \qquad \gamma' \coloneqq  \inf\{ t\geq0: \ V(X_t) = \phi(X_t) \},
    \end{align}
forms a NE for the game. This is even true for more general settings, e.g. (compare Benssousan \& Friedman \cite{Bens-Fried},\cite{Bens-Fried2}, Friedman \cite{zFz} and Karatzas \& Wang \cite{Kar}) when $E$ is replaced by $\R$ and $(X_t)_{t\geq 0}$ is a diffusion process,  under the conditions 
    \begin{align*}
        \forall x \in E, \quad \E_x \left[ \sup_{t \in \R^+} \mathrm{e}^{-\beta t} \phi(X_t) \right] < + \infty \quad \text{and} \qquad \lim_{t\rightarrow + \infty} \mathrm{e}^{-\beta t} \phi(X_t) = 0,
    \end{align*}
the game still has a value and the pair $(\tau',\gamma')$ defined in \eqref{intro: NE pair} remains a NE. In fact, the game has a value even when the integrability  condition $ \E_x \left[ \sup_{t \in \R^+} \mathrm{e}^{-\beta t} \phi(X_t) \right] < + \infty$ is removed but the pair $(\tau',\gamma')$ may not be a NE anymore, see Ekstr{\"o}m \& Villeneuve \cite{Erik}. For similar results in continuous settings, see e.g. Ekstr{\"o}m \& Peskir \cite{Erik-Peskir} and Peskir \cite{Peskir}. For an account of the general theories of optimal stopping problems and optimal stopping games, see the books by Shiryaev \cite{Shiryaev} and Peskir \& Shiryaev \cite{Pesk-Shir}.

In our settings, even if the game is guaranteed to have a value, the problem of computing this value function $V(x)$ is still hard. We can try to compute the function $ y \mapsto R_y(\tau',\gamma')$ but the sets $\{ x \in E: \ V(x) = \psi(x) \}$ (abbreviated $\{ V = \psi \}$) and $\{ x\in E: \ V(x)= \phi(x) \}$ (abbreviated $\{ V = \phi \}$) are both unknown. An alternative approach is to try plugging in all hitting times of fixed subsets of $E$ in the function $R_x(\cdot, \cdot) $ and check  if the criteria \eqref{intro: NE criteria} holds. Unfortunately, this is too costly since the number of subsets of $E$ is $2^{|E|}$, where $|E|$ is the cardinality of $E$, and thus grows exponentially fast with $|E|$. Also, even if we are lucky to guess correctly what $\{ V = \psi\}$ and $\{ V =\phi\}$ are, it is not easy to verify \eqref{intro: NE criteria}. This paper aims exactly at resolving this problem by proposing a new efficient algorithm to compute the value function $V$ as well as recover the critical sets $\{ V =\psi\}$ and $\{ V =\phi \}$. In essence, we extend the ideas of the "forward algorithm" in one-player games, initiated by Irle \cite{Irle} (for discrete-time Markov chains) and later revisited in Miclo \& Villeneuve \cite{MS} (for continuous-time Markov processes), to zero-sum two-player games. Our algorithm also recovers some known properties of $V$. For example, if $\psi$ is strictly smaller than $\phi$ everywhere, then it holds that
    \begin{align*}
        \begin{cases}
            \forall x \in \{ V = \psi \}, &\quad \Q[V](x) - \beta V(x) \leq 0 \\ 
            \forall x \in \{ V = \phi \}, &\quad \Q[V](x) - \beta V(x) \geq 0 \\
            \forall x \in \{ \psi < V < \phi \}, &\quad \Q[V](x) - \beta V(x) = 0 \label{intro: LV-bV = 0} \numberthis
        \end{cases}
    \end{align*}
where we have used the notation 
    \begin{align*}
        \forall f \in \R^E, \ \forall x \in E, \qquad \mathcal{Q}[f](x) \coloneqq  \sum_{y \in E} Q(x,y) f(y).
    \end{align*}
We quickly recall that the Markov generator $\mathcal{Q}= (Q(x,y))_{x,y \in E}$ is just a square matrix satisfying 
    \begin{align*}
        \forall x \in E, \quad \sum_{y \in E} Q(x,y) = 0 \qquad \text{and} \qquad \forall x \neq y, \quad Q(x,y) \geq 0.
    \end{align*}
The last equation in \eqref{intro: LV-bV = 0} is particularly useful because, once the critical sets $\{V=\psi\}$ and $\{V=\phi\}$ are known, the problem reduces to solving a linear system for $V$, which can be handled efficiently using standard numerical linear algebra methods. See the next section for details. A precise statement in the general case $\psi \leq \phi$ should read
    \begin{align*}
        \begin{cases}
            \forall x \in \{ V = \psi \} \setminus \{\phi = \psi \}, &\quad \Q[V](x) - \beta V(x) \leq 0 \\ 
            \forall x \in \{ V = \phi \} \setminus \{\phi = \psi \}, &\quad \Q[V](x) - \beta V(x) \geq 0 \\
            \forall x \in \{ \psi < V < \phi \}, &\quad \Q[V](x) - \beta V(x) = 0
        \end{cases}
    \end{align*}
so the sign of $ \Q[V](x) - \beta V(x)$ on $\{\phi = \psi\}$ is unknown but we can still solve for $V$ as long as we know the critical sets (see Lemmas \ref{lem: Q-beta = 0}, \ref{lem: critical obs} and Remark \ref{rem:compute} below). The number of steps taken before our algorithm terminates is strictly less than $|E|^2$, see Subsection \ref{subsec: aux} for more details. At each step, there is a system of linear equations to be solved and some conditions to check in order to obtain a sequence of shrinking sets, thus our algorithm has a similar flavor of "forward algorithm" by Irle \cite{Irle} and Miclo \& Villeneuve \cite{MS}. For completeness, since the computation of the function
    \begin{align} \label{citeV0}
        \forall x \in E, \quad V_0(x) \coloneqq  \sup_{\tau} \E_x[ \mathrm{e}^{-\beta \tau} \psi(X_\tau)]
    \end{align}
is needed in our algorithm, we will briefly recall the forward algorithm for one-player optimal stopping problem in Section \ref{subsec: forward for one player}. It is also well-known that $V_0$ is the smallest $\beta$-excessive function dominating $\psi$ (i.e. $V_0 \geq \psi$), where we recall that a function $f \in \R^E$ is $\beta$-excessive if $\Q[f](x) - \beta f(x) \leq 0$ for all $x \in E$. 

    We present a brief summary of our algorithm, outlining its key steps and core principles. To the best of our knowledge, there are no existing algorithms for computing the value function in the present setting, making our proposed algorithm entirely novel. First, we perform the "forward algorithm" to compute $V_0$ in \eqref{citeV0} and obtain the set $\{V_0 > \phi \} \coloneqq  \{x \in E: V_0(x) > \phi(x) \}$. We then set $S_1 = \{V_0 > \phi \} \cup \{ \phi = \psi \}$ and compute the next function
        \begin{align*}
           \forall x \in E, \quad V_1(x) \coloneqq  \sup_{\tau} R_x( \tau, \gamma_1), \qquad \text{where} \qquad \gamma_1 \coloneqq  \inf\{ t\geq0: X_t \in S_1 \}, 
        \end{align*}
    by adopting a "modified forward algorithm" (see Subsection \ref{Subsection for second case} for more details). Next, suppose we have obtained $V_k$ for some $k \geq 1$, we define the next set 
    \begin{align*} 
    S_{k+1} \coloneqq \Big( S_k \cap \{\Q[V_k] - \beta V_k \geq 0 \} \Big)\cup \{ \phi = \psi \},
    \end{align*}
and continue to compute the next function in the same manner
        \begin{align*}
           \forall x \in E, \  V_{k+1}(x) \coloneqq  \sup_{\tau} R_x( \tau, \gamma_{k+1}), \qquad \text{where} \qquad \gamma_{k+1} \coloneqq  \inf\{ t\geq0: X_t \in S_{k+1} \}.
        \end{align*}
    In this way, we obtain a decreasing sequence of functions $(V_k)_{k \geq 1}$ that converges pointwise to the value function $V$ after a finite number of iterations. The initial set $S_1$, its discovery, and the iterative shrinking of the sets $(S_k)_{k \geq 1}$ are at the core of this algorithm. They allow us to control the induced sequence $(V_k)_{k\geq 1}$ in such a way that $\psi \leq V_k \leq \phi$ for all $k \geq 1$, which aligns with our prior knowledge that $\psi \leq V \leq \phi$. We also remark that the set $S_1$ can be replaced by a slightly bigger set $\widetilde S_1 \coloneqq  \{ V_0 \geq \phi \}$ and the new induced sequence $(\widetilde V_k)_{k\geq 1}$ (defined in the same way as above) still decreases to $V$. For further details, see the discussions in Subsection \ref{subsec: aux}.
    
    Finally, we would like to emphasize that we only deal with finite state spaces because in practice we can effectively model and approximate real-world scenarios, such as those in finance, using  Markov processes taking values in finite (but possibly large) state spaces. This approach allows for more tractable analysis and computational implementation. Again, the main objective of this paper is to provide an algorithmic construction of the value function $V$ and the stopping regions (or the critical sets) $\{ V=\psi \}$ and $\{ V = \phi \}$. The structure of the paper is as follows. Section \ref{sec: settings} covers the settings, well-known results, notations and useful observations that are used frequently in later sections. This includes the basics of continuous-time Markov chains on finite state spaces, such as the Markov and strong Markov properties, Dynkin's martingale formula, as well as sufficient and necessary conditions for a pair of hitting times to be a NE. Section \ref{sec: algo} presents our algorithm in full detail. Subsection 3.1 offers a brief review of the forward algorithm by Irle \cite{Irle}, Miclo \& Villeneuve \cite{MS}, Subsection 3.2 provides the algorithms and proofs of convergence, and Subsection 3.3 discusses auxillary results and computational aspects. Section \ref{sec:examples} contains examples and applications for some well-known stochastic game in finance and economics. 

\section{Settings and frequently used results} \label{sec: settings}

Throughout this paper, $E$ is a finite state space, $X = (X_t)_{t\geq 0}$ is a continuous-time Markov chain defined on some probability space and takes values in $E$. We assume further that $X$ is time-homogeneous, has right-continuous sample paths and admits a Markov generator $\mathcal{Q}= (Q(x,y))_{x,y \in E}$. For $t\geq0$, let $\mathcal F_t \coloneqq  \sigma(X_s, 0\leq s \leq t)$ be its natural filtration and set $\mathcal F_\infty = \sigma(\cup_{t\geq 0} \mathcal F_t)$. Also, the probability space is assumed to be rich enough to facilitate a family of shift operators $(\theta_t)_{t\geq 0}$ satisfying $X_t \circ \theta_s = X_{t+s}$ for any $t,s \geq 0$, and on this probability space, there exists a family of probabilities $(\mathbb P_x)_{x \in E}$ such that $ \mathbb P_x(X_0 = x) = 1$. For $t \geq 0$, we define the transition matrix $\mathcal P_t = (\mathcal P_t(x,y))_{x,y \in E}$ by
    \begin{align*}
        \forall x,y \in E,  \qquad \mathcal P_t(x,y) \coloneqq \mathbb P_x(X_t=y),
    \end{align*}
then it is well-known that (see e.g. Norris \cite{Norris}) $t \mapsto \mathcal P_t$ satisfies the differential and integral equations (in matrix form)
    \begin{align*}
         \qquad \mathcal P_0 = I, \quad \forall t >0,\quad  \partial_t \mathcal P_t = \mathcal P_t\mathcal{Q}= \mathcal{Q}\mathcal P_t \iff \forall t \geq 0, \quad \mathcal P_t = I + \int^t_0 \mathcal P_s \mathcal{Q}ds.
    \end{align*}
Therefore, for any discounted rate $\beta > 0$, we differentiate in $t$ the quantity $\mathrm{e}^{-\beta t}\mathcal P_t$ and integrate back again, we get 
    \begin{align*} 
        \forall t\geq0, \quad \mathrm{e}^{-\beta t}\mathcal P_t = I + \int^t_0 \mathrm{e}^{-\beta s}(\Q - \beta I)\mathcal P_s ds.
    \end{align*}
As a consequence, by taking expectation under $\mathbb P_x$, we obtain the process $(M^x_t)_{t\geq 0}$, defined by
    \begin{align} \label{sec2: martingale}
        M^x_t \coloneqq  \mathrm{e}^{-\beta t} f(X_t) - f(x) - \int^t_0 \mathrm{e}^{-\beta s} \Big(\mathcal \mathcal{Q}[f](X_s) -\beta f(X_s)\Big) ds,
    \end{align}
is a bounded $((\mathcal F_t)_{t\geq 0}, \mathbb P_x)$-martingale for any $f \in \R^E$ and $x \in E$ (which is no longer true if $E$ is not finite).

A random time $\eta$, taking values in $[0, +\infty]$, is called a stopping time if $\{ \eta \leq t \} \in \mathcal F_t$  for all $t\geq 0$. The filtration associated with the stopping time $\eta$ is $\mathcal F_\eta \coloneqq  \{ A \in \mathcal F_\infty: A \cap \{\eta \leq t\} \in \mathcal F_t \}$. The strong Markov property then states that 
    \begin{align} \label{strongMarkov}
        \forall x \in E, \quad \mathbb E_x [Y \circ \theta_\eta \mathbf{1}_{ \{\eta < +\infty\} } | \mathcal F_\eta]  = E_{X_\eta}[Y]  \mathbf{1}_{\{\eta < +\infty \} }
    \end{align}
for any bounded, $\mathcal F_\infty$-measurable function $Y$ and stopping time $\eta$. By the optional sampling theorem, we have from \eqref{sec2: martingale} that for any stopping time $\eta$, 
    \begin{align*}
        \forall t\geq0, \ \forall x\in E , \quad \E_x[\mathrm{e}^{-\beta \eta \wedge t} f(X_{\eta\wedge t})] = f(x) + \E_x \left[ \int^{\eta \wedge t}_0 \mathrm{e}^{-\beta s} \Big(\mathcal \mathcal{Q}[f](X_s) -\beta f(X_s)\Big) ds \right].
    \end{align*}
Thus, by letting $t\rightarrow + \infty$, we get from bounded convergence theorem the so-called Dynkin's formula (DF)
    \begin{align} \label{intro: DF}
        \forall x \in E, \quad \E_x[\mathrm{e}^{-\beta \eta} f(X_\eta) \mathbf{1}_{ \{\eta < +\infty \} }] = f(x) + \E_x \left[ \int^{\eta}_0 \mathrm{e}^{-\beta s} \Big(\mathcal \mathcal{Q}[f](X_s) -\beta f(X_s)\Big) ds\right].
    \end{align}
If we agree to use the convention that  $\mathrm{e}^{-\eta} f(X_\eta) = 0$ on $ \{\eta = +\infty \} $, we can rewrite above equation as 
    \begin{align*}
        \quad \E_x[\mathrm{e}^{-\beta \eta} f(X_\eta)] = f(x) + \E_x \left[ \int^{\eta}_0 \mathrm{e}^{-\beta s} \Big(\mathcal \mathcal{Q}[f](X_s) -\beta f(X_s)\Big) ds \right].
    \end{align*}

Next, we introduce some frequently used notations. For any $A \subset E$, we denote 
    \begin{align*}
        h(A) \coloneqq  \inf \{ t\geq 0: \ X_t \in A \}
    \end{align*}
as the first hitting time of the fixed set $A$. This notation will be used repeatedly in this paper to lighten the burden of too much notations. For any two functions $f,f' \in \R^E$, we denote by $\{ f = f' \}$ the set $\{x \in E: f(x) = f'(x) \} $. Similarly, the sets $\{ f \leq f' \}$, $\{ f \geq f' \}$ and others along these lines are defined analogously. For example, $\{\Q[f] - \beta f \geq 0 \} $ is the set $\{ x \in E: \Q[f](x) - \beta f(x) \geq 0 \}$. We also write "$f \bullet f'$", where $\bullet \in \{ >, <, = \}$, to mean that $f(x) \bullet f'(x)$ for all $x \in E$ and if we write "$f \bullet f'$ on $G$", where $G$ is a subset of $E$, we mean that $f(x) \bullet f'(x)$ for all $x \in G$.

The following useful lemmas are direct consequences of the strong Markov property of $X$.
	\begin{lemma} \label{SMP}
		Let $B \subset C$ be two subsets of $E$, and let $\sigma$ denote the first jump time of $X$. Suppose that a function $g\in \R^E$ has the property 
		\begin{align} \label{func:SMP}
		\forall x \in E, \quad g(x) = \E_x[\mathrm{e}^{-\beta h(B)} g(X_{h(B)}) \mathbf{1}_{ \{ h(B) < + \infty\}}]. 
		\end{align}
		Then,
		\[ \forall x \notin C, \quad g(x) = \E_x[\mathrm{e}^{-\beta \sigma} g(X_{\sigma}) \mathbf{1}_{\{\sigma < +\infty \}}] =\E_x[\mathrm{e}^{-\beta h(C)} g(X_{h(C)}) \mathbf{1}_{ \{ h(C) < + \infty\}}].  \]
	\end{lemma}
	\begin{proof}
		Let $x \notin C$. Observe that 
			\begin{align*}
				h(B) = h(C) + h(B) \circ \theta_{h(C)} \quad \text{and}\quad  h(B)= \sigma+ h(B)\circ \theta_\sigma \quad \mathbb P_x-\text{a.s.},
			\end{align*}
		where for any two stopping times $\eta,\zeta$, the expression $\eta = \zeta + \eta \circ \theta_\zeta$ means 
			\[ \forall \omega \in \Omega, \quad \eta(\omega) = \zeta(\omega) + \eta( \theta_{\zeta(\omega)}(\omega)) =  \zeta(\omega) + \eta(\omega_{\zeta(\omega) + \cdot}). \]
		We have
			\begin{align*}
        		g(x) &= \E_x[\mathrm{e}^{-\beta h(B)} g(X_{h(B)}) \mathbf{1}_{\{h(B) < +\infty \}} ] \\ 
        &= \E_x \left[ \mathrm{e}^{-\beta \sigma } \mathbf{1}_{\{\sigma < +\infty \}} E_x \left[\mathrm{e}^{-\beta (h(B) \circ \theta_\sigma)} g(X_{h(B)} \circ \theta_\sigma) \mathbf{1}_{\{h(B) \circ \theta_\sigma < +\infty \}} | \mathcal F_\sigma \right] \right] \\ 
        &=  \E_x[ \mathrm{e}^{-\beta \sigma } \mathbf{1}_{\{\sigma < +\infty \}} E_{X_\sigma} [\mathrm{e}^{-\beta h(B)} g(X_{h(B)}) \mathbf{1}_{\{h(B) < +\infty \}} ] ]  \quad (\text{use \eqref{strongMarkov}}) \\
        &= \E_x[\mathrm{e}^{-\beta \sigma} g(X_{\sigma}) \mathbf{1}_{\{\sigma < +\infty \}}] \quad (\text{use \eqref{func:SMP}}) .
        			\end{align*}
	In the same manner, we get $g(x) =\E_x[\mathrm{e}^{-\beta h(C)} g(X_{h(C)}) \mathbf{1}_{ \{h(C) < + \infty \}}] $, which completes the proof.
	\end{proof}

\begin{lemma} \label{lem: Q-beta = 0}
    Let $B, C \subset E$ be disjoint. Define a function $g \in \R^E$ by 
    \begin{align*}
        \forall x \in E, \quad g(x) \coloneqq  R_x(h(B),h(C)).
    \end{align*}
Then for all $x \notin B \cup C$,  $\Q[g](x) - \beta g(x)= 0$, $g = \psi$ on $B$ and $g = \phi$ on $C$. Conversely, if there is a function $g^*$ satisfying the previous conditions, then $g^* = g$.
\end{lemma}
\begin{proof}
    The fact that $g = \psi$ on $B$ and $g = \phi$ on $C$ are straightforward. From this, we can express $g$ as
        \begin{align*} 
            \forall x \in E, \quad g(x) = \E_x[\mathrm{e}^{-\beta \eta'}g(X_{\eta'})] = \E_{x} [\mathrm{e}^{-\beta \eta'} g(X_{\eta'}) \mathbf{1}_{\{\eta' < +\infty \}} ],
        \end{align*}
    where $\eta'\coloneqq   h(B) \wedge h(C) = h(B\cup C)$. For any $ x \notin B \cup C$, applying Lemma \ref{SMP} yields
    \begin{align*}
        g(x) &= \E_x[\mathrm{e}^{-\beta \sigma} g(X_{\sigma}) \mathbf{1}_{\{\sigma < +\infty \}}] \\
        &= g(x) +  \E_x \left[ \int^{\sigma}_0 \mathrm{e}^{-\beta s} \Big(\mathcal \Q[g](X_s) -\beta g(X_s)\Big) ds \right] \\
        &= g(x) + \Big( \Q[g](x) - \beta g(x)\Big) \E_x \left[ \frac{1 - \mathrm{e}^{-\beta \sigma}}{\beta} \right],
    \end{align*}
which implies the claim because $E_x[\mathrm{e}^{-\beta \sigma}] < 1$ (recall that $\sigma$ is exponentially distributed with rate $|Q(x,x)|$ so $\sigma >0$ $\mathbb P_x$-a.s.).

For the converse, we use Dynkin's formula \eqref{intro: DF} to get
    \begin{align*}
        \forall x \in E, \quad g(x) &= \E_x[\mathrm{e}^{-\beta h(B)}\psi(X_{h(B)}) \mathbf{1}_{ \{h(B) \leq h(C) \}} + \mathrm{e}^{-\beta h(C)}\phi(X_{h(C)}) \mathbf{1}_{ \{h(B) >h(C) \}} ] \\ 
        &= \E_x[\mathrm{e}^{-\beta h(B)}g^*(X_{h(B)}) \mathbf{1}_{ \{h(B) \leq h(C)\}} + \mathrm{e}^{-\beta h(C)}g^*(X_{h(C)}) \mathbf{1}_{ \{h(B) >h(C) \}} ] \\
        &= g^*(x) + \E_x \left[ \int^{h(B)\wedge h(C)}_0 \mathrm{e}^{-\beta s} \Big(\mathcal \Q[g^*](X_s) -\beta g^*(X_s)\Big) ds \right] \\
        &= g^*(x)
    \end{align*}
where the last equality follows since $\Q[g^*] -\beta g^* = 0$ outside $B\cup C$. This concludes the proof.
\end{proof}
\begin{rem} \label{rem:compute}
    Lemma \ref{lem: Q-beta = 0} provides us a way to compute the function $g(x) = R_x(h(B), h(C))$ numerically by solving the associated linear system (e.g., using standard linear system solvers such as LU decomposition or iterative methods). To see this, we define a new generator $\Q' = (Q'(x,y))_{x,y \in E}$ by 
        \begin{align*}
            \forall x \in B \cup C, \ \forall y \in E, \qquad Q'(x,y) \coloneqq  0, \qquad \forall x \notin B \cup C, \ \forall y \in E, \qquad Q'(x,y) \coloneqq  Q(x,y).
        \end{align*}
    That is, we set all rows of $\Q$ corresponding to states in $B \cup C$ equal to 0, while leaving the remaining rows unchanged. One can then verify that $g$ satisfies the linear system $(\Q' -\beta I) g = v$, where $v$ is the vector satisfying $v(x) = -\beta \phi(x)$ for $x \in C$, $v(x) = -\beta \psi (x)$ for $x \in B$ and $v(x) = 0$ for $x \notin B \cup C$. Since $ \mathcal L - \beta I$ is always invertible for every $\beta >0$ and any Markov generator $\mathcal L$, the unique solution  $g = (\Q'-\beta I)^{-1} v$ can be computed numerically by solving this linear system. To see why $ \mathcal L - \beta I$ is invertible, consider the stochastic matrix 
    	\[ P = \frac{1}{L}\mathcal L+ I, \quad \text{where} \quad L = \max_{x \in E} | \mathcal L (x,x) | > 0.\]
    Now, suppose that $\mathcal L[f] - \beta f =0 $ for some non-zero function $f$, then
    	\[ P[f] = (\frac{\beta}{L} + 1)f,\]
	implying that $P$ has an eigenvalue strictly greater than 1, which is a contradiction.
\end{rem}

The following property of the expected payoff $R_x(\cdot,\cdot)$ is also useful.
    \begin{lemma} \label{lem: critical obs}
         Let $A, B \subset E$ be two disjoint sets and satisfies $A \cap \{ \phi = \psi \} = B \cap \{ \phi = \psi \} = \emptyset $. Then it holds that
            \begin{align*}
                \forall x \in E, \qquad R_x \Big(h(A \cup \{ \phi = \psi \}), h(B \cup \{ \phi = \psi \} )\Big)= R_x \Big(h(A), h(B \cup \{ \phi = \psi \})\Big).
            \end{align*}
    \end{lemma}
    \begin{proof}
        Let $C: = \{ \phi = \psi\}$. We first observe that 
            \begin{align*}
                \Big\{ h(A \cup C) \leq h(B \cup C ) \Big\} &= \Big\{ h(A) \wedge h(C) \leq h(B \cup C ) \Big\} \\
                &= \Big\{ h(A) \leq h(B \cup C ) \Big\} \cup \Big\{ h(C) \leq h(B \cup C ) \Big\} \\
                &= \Big\{ h(A) \leq h(B \cup C ) \Big\} \cup \Big\{ h(C) = h(B \cup C ) \Big\} \\ 
                &= \underbrace{\Big\{ h(A) \leq h(B \cup C ) \Big\}}_{\coloneqq  \Omega_1} \cup \underbrace{\Big(\Big\{ h(C) = h(B \cup C ) \Big\} \cap  \Big\{h(A) > h(B \cup C ) \Big\} \Big)}_{\coloneqq  \Omega_2}.
            \end{align*}
    For all $x \in E$,
        \begin{align*}
           R_x \Big(h(A \cup C) , h(B \cup C )\Big)&= \E_x \left[ \mathrm{e}^{-\beta h(A \cup C)} \psi (X_{h(A \cup C)}) \mathbf{1}_{\Omega_1 \cup \Omega_2}  + \mathrm{e}^{-\beta h(B \cup C)} \phi (X_{h(B \cup C)}) \mathbf{1}_{\Omega_1^c \cap  \Omega_2^c} \right] \\
           &= \E_x \Big[ \mathrm{e}^{-\beta h(A)} \psi (X_{h(A)}) \mathbf{1}_{\Omega_1} + \mathrm{e}^{-\beta h(C)} \psi (X_{h(C)}) \mathbf{1}_{ \Omega_2}  \\
           & \hspace{5cm} + \mathrm{e}^{-\beta h(B \cup C)} \phi (X_{h(B \cup C)}) \mathbf{1}_{\Omega_1^c \cap  \Omega_2^c} \Big] \\
           &= \E_x\ Big[ \mathrm{e}^{-\beta h(A)} \psi (X_{h(A)}) \mathbf{1}_{\Omega_1} + \mathrm{e}^{-\beta h(C)} \phi (X_{h(C)}) \mathbf{1}_{ \Omega_2} \\
           & \hspace{5cm}  + \mathrm{e}^{-\beta h(B \cup C)} \phi (X_{h(B \cup C)}) \mathbf{1}_{\Omega_1^c \cap  \Omega_2^c} \Big] \\
           &= \E_x \Big[ \mathrm{e}^{-\beta h(A)} \psi (X_{h(A)}) \mathbf{1}_{\Omega_1} + \mathrm{e}^{-\beta h(B\cup C)} \phi (X_{h(B\cup C)}) \mathbf{1}_{ \Omega_2}  \\
           & \hspace{5cm}  + \mathrm{e}^{-\beta h(B \cup C)} \phi (X_{h(B \cup C)}) \mathbf{1}_{\Omega_1^c \cap  \Omega_2^c} \Big] \\
           &= \E_x \left[ \mathrm{e}^{-\beta h(A)} \psi (X_{h(A)}) \mathbf{1}_{\Omega_1}  + \mathrm{e}^{-\beta h(B \cup C)} \phi (X_{h(B \cup C)}) \mathbf{1}_{\Omega_1^c} \right] \\
           &= R_x \Big(h(A) , h(B \cup C )\Big),
        \end{align*}
    where we have used $ \psi (X_{h(C)}) = \phi (X_{h(C)}) $ on $\Omega_2$ for the third equality and $\Omega_2 \subset \Omega_1^c$ to obtain $(\Omega_1^c \cap \Omega_2^c) \cup \Omega_2 = \Omega_1^c$ in the fifth equality.
    \end{proof}

    The following observation provides a sufficient and necessary condition to check if a pair $(\tau,\gamma)$ is a NE. It is a standard verification result but we give a proof for the sake of completeness. The proof of sufficiency is provided here while the necessity part is obtained in Corollary \ref{cor: necessity part} as a consequence of the Theorems \ref{theo: big seq} and \ref{theo: main}.
        \begin{theorem} \label{theo: sufficient for NE}
            Suppose there are two disjoint sets $A,B$, each is disjoint from $\{ \phi = \psi \}$, and a function $\widetilde V$ satisfying
            \begin{align} \label{conditions for a NE}
                \begin{cases}
                    \forall x \in A, \quad &\Q[\widetilde V](x) -\beta \widetilde V(x) \leq 0 \ \text{and} \ \widetilde V(x) = \psi(x) \\
                    \forall x \in B, \quad &\Q[\widetilde V](x) -\beta \widetilde V(x) \geq 0 \ \text{and} \ \widetilde V(x) = \phi(x) \\
                    \forall x \notin A\cup B \cup \{\phi = \psi\}, \quad &\Q[\widetilde V](x) -\beta \widetilde V(x) = 0 \\
                    \forall x \in E, \quad &\psi(x) \leq \widetilde V(x) \leq \phi(x)
                \end{cases}
            \end{align}
            then $\Big(h \big(A \cup \{ \phi =\psi\}\big), h\big(B \cup \{ \phi = \psi \}\big) \Big)$ forms a NE and $\widetilde V = V$, where $V$ is the value function.
        \end{theorem}
        \begin{proof}
            Denote $(\tau_A,\tau_B) \coloneqq \Big(h \big(A \cup \{ \phi =\psi\}\big), h\big(B \cup \{ \phi = \psi \}\big) \Big)$. Applying Lemma \ref{lem: Q-beta = 0} for the disjoint sets $A$ and $B \cup \{ \phi = \psi \}$, we get 
            \begin{align*}
            	\widetilde V(x) &= R_x \Big(h(A) , h(B \cup  \{ \phi = \psi \} )\Big)  = R_x(\tau_A,\tau_B),
            \end{align*}
            where the last equality follows from Lemma \ref{lem: critical obs}. For all stopping times $\tau$ and $x \in E$, we have
                \begin{align*}
                    R_x(\tau, \tau_B) &= \E_x[\mathrm{e}^{-\beta \tau}\psi(X_\tau) \mathbf{1}_{\{\tau \leq \tau_B\}} + \mathrm{e}^{-\beta \tau_B}\phi(X_{\tau_B})
                    \mathbf{1}_{ \{\tau >\tau_B \} } ] \\
                    &= \E_x[\mathrm{e}^{-\beta \tau}\psi(X_\tau) \mathbf{1}_{\{\tau \leq \tau_B\}} + \mathrm{e}^{-\beta \tau_B}\widetilde V(X_{\tau_B})
                    \mathbf{1}_{ \{\tau >\tau_B \} } ] \quad (\text{since} \ \widetilde V=\phi \ \text{on} \ B) \\
                    &\leq \E_x[\mathrm{e}^{-\beta \tau}\widetilde V(X_\tau) \mathbf{1}_{\{\tau \leq \tau_B\}} + \mathrm{e}^{-\beta \tau_B}\widetilde V(X_{\tau_B})
                    \mathbf{1}_{ \{\tau >\tau_B \} } ] \quad (\text{since} \ \psi \leq \widetilde V) \\
                    &=\E_x[\mathrm{e}^{-\beta (\tau\wedge\tau_B)}\widetilde V(X_{\tau\wedge \tau_B})] \\
                    &= \widetilde V(x) + \E_x\left[ \int^{\tau\wedge \tau_B}_0 \mathrm{e}^{-\beta s} \Big(  \Q[\widetilde V](X_s) -\beta \widetilde V(X_s)\Big) ds \right] \quad (\text{DF}) \\
                    & \leq \widetilde V(x) = R_x(\tau_A,\tau_B)
                \end{align*}
            because $\Q[\widetilde V] -\beta \widetilde V \leq 0$ outside $B \cup \{\phi = \psi\}$. Thus, we have 
            	\begin{align*}
            		\widetilde V(x) = R_x(\tau_A,\tau_B) = \sup_\tau R_x(\tau,\tau_B),
            	\end{align*}
            and similarly, 
            	\begin{align*}
            		\widetilde V(x) = R_x(\tau_A,\tau_B) = \inf_\gamma R_x(\tau_A,\gamma).
            	\end{align*}
	Hence, the pair $(\tau_A,\tau_B)$ is indeed a NE (cf. \eqref{intro: NE criteria}) and therefore $\widetilde V=V$.
        \end{proof}

\section{The algorithm} \label{sec: algo}
    \subsection{Forward algorithm for one-player game} \label{subsec: forward for one player}
    We recall that a function $f \in \R^E$ is called $\beta$-excessive (with respect to $\Q$) if it satisfies
        \begin{align*}
            \forall x \in E, \qquad \Q[f](x) - \beta f(x) \leq 0.
        \end{align*}
    It is well known the the function
    \begin{align} \label{func: V0}
        \forall x \in E, \quad V_0(x) \coloneqq  \sup_{\tau} \E_x[\mathrm{e}^{-\beta \tau} \psi(X_{\tau})] 
    \end{align}
    is the smallest $\beta$-excessive function dominating $\psi$, which clearly coincides with $\sup_{\tau} R_x(\tau, + \infty)$. It can be computed using the following "forward algorithm", which we adapt from Miclo \& Villeneuve \cite{MS}. Firstly, define
        \begin{align*} 
            C_1 = \{ \Q[\psi] - \beta \psi \leq 0 \}, \qquad \eta_1\coloneqq  h(C_1), \qquad V^{(1)}_0(x) \coloneqq  \E_x[\mathrm{e}^{-\beta \eta_1} \psi (X_{\eta_1})].
        \end{align*}
    Suppose we have defined the triple $(C_n, \eta_n, V^{(n)}_0)$ for some $n \geq 1$, we define the next triple $(C_{n+1}, \eta_{n+1}, V^{(n+1)}_0)$ by 
        \begin{align} \label{V^n_0}
            C_{n+1} = C_n \cap \{ \Q[V^{(n)}_0] - \beta V^{(n)}_0 \leq 0 \}, \qquad \eta_{n+1}\coloneqq  h(C_{n+1}), 
        \end{align}
       and 
        \begin{align} \label{V^n_0 def}
        	\forall x \in E, \quad V^{(n+1)}_0(x) \coloneqq  \E_x[\mathrm{e}^{-\beta \eta_{n+1}} \psi (X_{\eta_{n+1}})].
        \end{align}
    Then we obtain the following result, whose proof can be found in Miclo \& Villeneuve \cite{MS}.
    \begin{theorem} \label{theo: recall}
        The sequence $(V^{(n)}_0)_{n \geq 1}$ is increasing and converges pointwise to $V_0$ after a finite number of steps and $V_0$ is the smallest $\beta$-excessive function dominating $\psi$. Moreover, we have $\{ V_0 = \psi \} = \bigcap_{n \geq 1} C_n$  and that 
    \begin{enumerate}
        \item For all $x \in \{ V_0 = \psi \}$, $\Q[V_0](x) - \beta V_0(x) \leq 0$.
        \item For all $x \notin  \{ V_0 = \psi \}$,  $\Q[V_0](x) - \beta V_0(x) =0$.
        \item The stopping time $\eta^* = h(\{ V_0 = \psi\})$ is optimal, i.e. $\forall x \in E$, $V_0(x) = \mathbb E_x[\mathrm{e}^{-\beta \eta^*} \psi(X_{\eta^*})]$.
    \end{enumerate}
    \end{theorem}
    
    A consequence of this approach, which was not given in \cite{MS}, is that
        \begin{theorem}\label{ends after finite time}
            We have $\forall x \in E$, $\mathbb P_x( h( \{ V_0  = \psi\}) < + \infty) = 1$.
        \end{theorem}
        \begin{proof}
            Since $E$ is finite, it is well-known that there exist recurrent classes $F_1,...,F_d$ ($d \geq 1$) such that
                \begin{align*}
                    \forall x \in E, \quad \mathbb P_x \Big( \exists t_0 \geq 0 \ \text{such that} \ \forall t \geq t_0, \quad X_t \in \bigcup_{i=1}^d F_i \Big) = 1,
                \end{align*}
            or equivalently, $X_t \in \bigcup_{i=1}^d F_i$ eventually $\mathbb P_x$-a.s. for every $x \in E$. Define
                \begin{align*}
                    \forall i = 1,...,d, \quad W_i \coloneqq  \argmax_{F_i} \psi, \qquad W \coloneqq  \bigcup_{i=1}^d W_i.
                \end{align*}
            We shall prove by induction that $W \subset C_n$ for all $n \geq 1$. Clearly, $W \subset \{ \Q[\psi] - \beta \psi \leq 0 \} = C_1$. Suppose this is true for some $n \geq 1$, we prove it is true for $n+1$. Observe from $\eqref{V^n_0}$, \eqref{V^n_0 def} that
                \begin{align*}
                    \forall i = 1,...,d, \ \forall x \in W_i, \qquad V^{(n)}_0(x) = \psi(x) = \max_{F_i} \psi = \max_{F_i} V^{(n)}_0.
                \end{align*}
            For all $y \in F_i$,  $V^{(n)}_0 (y) \leq \max_{F_i} \psi =  \psi(x_0) = V^{(n)}_0 (x_0)$ for some $x_0 \in W_i \subset C_n$ (previous induction step)  because if we start at some $y \in F_i$, the process stays in $F_i$ forever. This implies 
                \begin{align*}
                    \forall i = 1,...,d, \ \forall x \in W_i, \qquad \Q[V^{(n)}_0](x) -  \beta V^{(n)}_0(x) \leq 0,
                \end{align*}
            and hence $W \subset C_n \cap \{ \Q[V^{(n)}_0] - \beta V^{(n)}_0 \leq 0 \} = C_{n+1}$, finishing the induction. This gives $W \subset \bigcap_{n \geq 1} C_n =  \{ V_0 = \psi \}$. Since $\forall x \in E$, $\mathbb P_x( h (W) < +\infty) = 1$, we must have $\forall x \in E$, $\mathbb P_x( h(\{ V_0  = \psi\}) < + \infty) = 1$ as claimed.      
        \end{proof}
\subsection{First case: $\phi \geq V_0$}

The next result shows why we can compute the value function $V$ if  $\phi \geq V_0$.
    \begin{theorem} \label{theo: V co V0}
        It holds that $V = V_0$ and the pair of hitting times $\Big(h(\{ V_0 = \psi \}), h( \{ V_0 = \phi \}) \Big)$ forms a NE. 
    \end{theorem}
    \begin{proof}
        We use Theorem \ref{theo: sufficient for NE} here. Let $A\coloneqq  \{ V_0 = \psi \} \cap \{ \phi > \psi \}$ and $B\coloneqq  \{ V_0 = \phi \} \cap \{ \phi > \psi \}$. By Theorem \ref{theo: recall}, we have $ \Q[V_0] - \beta V_0 \leq 0$ on $A$ and $\Q[V_0] - \beta V_0 =0$ on $B$. Note that $A,B$ and $\{ \phi = \psi \}$ are pairwise disjoint and $\{ V_0 = \psi \} \subset A\cup B \cup \{\phi =\psi \}$. By Theorem \ref{theo: recall} again, $\Q[V_0] - \beta V_0 =0$ outside $A \cup B \cup \{ \phi = \psi \}$. Thus, all the conditions in Theorem \ref{theo: sufficient for NE} are satisfied and the result follows.
    \end{proof}
\begin{rem}
    We have to emphasize that NE are not unique. For example, when $\phi = V_0 = V$, then $\forall \gamma_0 \in \mathcal M$, the pair $(h(\{ V = \psi \}), \gamma_0)$ is also a NE. Indeed, it is easy to see that
        \begin{align*}
            \forall x \in E, \ \forall \gamma \in \mathcal M, \quad R_x(h(\{ V = \psi \}), \gamma) = V_0(x) = V(x).
        \end{align*}
    Also, $\forall (x,\tau, \gamma) \in E \times \mathcal M^2$, $R_x(\tau,\gamma) \leq V_0(x)$ because $\psi \leq \phi = V_0$ and $V_0$ is $\beta$-excessive. As a consequence, we get 
        \begin{align*}
           \forall \tau, \gamma \in \mathcal M, \qquad R_x(\tau, \gamma_0) \leq  R_x(h(\{ V = \psi \}), \gamma_0) \leq R_x(h(\{ V = \psi \}), \gamma),
        \end{align*}
    and therefore $(h(\{ V = \psi \}), \gamma_0)$ forms a NE (cf. \eqref{intro: NE criteria}).
\end{rem}

Theorem \ref{theo: V co V0} says that, if $\phi \geq V_0$ then $V = V_0$, which allows us to use the forward algorithm introduced in the previous Subsection \ref{subsec: forward for one player} to compute $V_0 = V$. Since we are working with finite state spaces, the following sufficient conditions could help us quickly verify $\phi \geq V_0$ (the last two bullet points do not require that $E$ is finite): 
        \begin{itemize}
            \item $\min \phi \geq \max \psi$, since it holds that for any $x\in E$, $V_0(x) \leq \max \psi$ from definition \eqref{func: V0},
            \item $\psi$ is $\beta$-excessive, since it implies $V_0 = \psi$ (recalling from Theorem \ref{theo: recall} that $V_0$ is the smallest $\beta$-excessive function dominating $\psi$) and so $V_0 \leq \phi$,
            \item $\phi$ is $\beta$-excessive, since it implies $\phi \geq V_0$ because $V_0$ is the smallest $\beta$-excessive function dominating $\psi$.
        \end{itemize}
        
    We now move on to tackle the general case, which is the main objective of this paper.

\subsection{Second case: \texorpdfstring{$\{ V_0 > \phi \} \neq \emptyset$}{Lg} } \label{Subsection for second case}
To find the value function $V$ in this case, we proceed as follows. Set 
        \begin{equation} \label{def S_1}
            S_1 \coloneqq \{ V_0 > \phi \} \cup \{ \phi = \psi \}, \quad \gamma_1\coloneqq  h(S_1) \qquad \text{and}  \quad V_1(x) \coloneqq  \sup_{\tau} R_x(\tau, \gamma_1).
        \end{equation}           
    To be able to compute $V_1$, we shall recursively define a sequence of functions that increases to $V_1$ after a finite number of iterations. Firstly, we define
        \begin{align} \label{eq: V^1_1}
            D^{(1)}_1 \coloneqq  \{ \Q[\psi] - \beta \psi \leq 0 \} \cap S^c_1, \quad  \tau^{(1)}_1 \coloneqq  h (D^{(1)}_1) \qquad \text{and} \quad V^{(1)}_1(x)\coloneqq  R_x(\tau^{(1)}_1, \gamma_1).
        \end{align}
    Next, suppose we have defined $(D^{(n)}_1, \tau^{(n)}_1 , V^{(n)}_1)$ for some $n \geq 1$, we define the next triple $(D^{(n+1)}_1, \tau^{(n+1)}_1 , V^{(n+1)}_1 )$ as 
        \begin{align} \label{eq: V^(n)_1}
            D^{(n+1)}_1\coloneqq  \{ \Q[V^{(n)}_1] - \beta V^{(n)}_1  \leq 0 \} \cap D^{(n)}_1, \quad \tau^{(n+1)}_1 \coloneqq  h(D^{(n+1)}_1) \quad \text{and} \quad V^{(n+1)}_1(x)\coloneqq  R_x(\tau^{(n+1)}_1, \gamma_1).
        \end{align}

    We have the following observation 
    \begin{lemma} \label{prop: obs 1}
        For any subset $D\subset E$ such that $D \cap S_1 = \emptyset$. Then it holds that 
            \begin{align*}
                \forall x \in  E,  \qquad u_D(x) \coloneqq  R_x(h(D), \gamma_1) \leq V_0(x) \wedge \phi(x).
            \end{align*}
        In particular, we have $\forall n \geq 1$, $V^{(n)}_1 \leq V_0 \wedge \phi$.
    \end{lemma}
    \begin{proof}
        By the definition of $u_D$, we have $u_D(x) = \phi(x) $ on the set $S_1$ thus either $u_D(x) = \phi(x) < V_0(x) $ or $ u_D(x) = \phi(x) = \psi(x) \leq V_0$. For $x \notin S_1$, by Dynkin's formula \eqref{intro: DF}
            \begin{align*}
                u_D(x) = R_x(h(D), \gamma_1) &= \E_x[\mathrm{e}^{-\beta h(D)}\psi(X_{h(D)}) \mathbf{1}_{ \{h(D) \leq \gamma_1\}} + \mathrm{e}^{-\beta \gamma_1}\phi(X_{\gamma_1}) \mathbf{1}_{\{ h(D) >\gamma_1\}} ] \\
                &\leq  \E_x[\mathrm{e}^{-\beta h(D)}V_0(X_{h(D)}) \mathbf{1}_{\{ h(D) \leq \gamma_1\}} + \mathrm{e}^{-\beta \gamma_1}V_0(X_{\gamma_1}) \mathbf{1}_{\{h(D) >\gamma_1\}} ]  \\ 
                &\leq \E_x[\mathrm{e}^{-\beta (h(D) \wedge \gamma_1)} V_0(X_{h(D) \wedge \gamma_1})] \\
                &\leq V_0(x) + \E_x\left[ \int^{h(D) \wedge \gamma_1}_0 \mathrm{e}^{-\beta s}\Big(\Q[V_0](X_s) - \beta V_0(X_s) \Big) ds \right] \\
                &\leq V_0(x)
            \end{align*}
    because $V_0$ is $\beta$-excessive.
    \end{proof}
    \begin{lemma} \label{cor: direct of prop obs 1}
        We have $V^{(1)}_1 \geq \psi$, and the strict inequality $ V^{(1)}_1 > \psi $ holds outside the set \[ D^{(1)}_1 \cup S_1  = \{ \Q[\psi] - \beta \psi \leq 0 \} \cup \{V_0 > \phi \} \cup \{ \phi = \psi \}.\]
    \end{lemma}
    \begin{proof}
        For all $x \notin D^{(1)}_1 \cup S_1$, we have
            \begin{align*}
                V^{(1)}_1(x) &\geq \E_x\Big[\mathrm{e}^{-\beta (\tau^{(1)}_1 \wedge \gamma_1) } \psi(X_{\tau^{(1)}_1 \wedge \gamma_1}) \Big] \\
                &= \psi(x) + \E_x\left[ \int^{\tau^{(1)}_1 \wedge \gamma_1}_0 \mathrm{e}^{-\beta s}\Big(\Q[\psi](X_s) - \beta \psi(X_s) \Big) ds \right] \\
                &> \psi(x),
            \end{align*}
        because $ \tau^{(1)}_1 \wedge \gamma_1 = h (D^{(1)}_1 \cup S_1) >0$ $\mathbb P_x$-a.s.  and $ \Q[\psi](X_s) - \beta \psi(X_s) > 0 $ before the time $\tau^{(1)}_1 \wedge \gamma_1$. Finally, for $x \in D^{(1)}_1 \cup S_1 $, the value $V^{(1)}_1(x)$ is either $\psi(x)$ or $\phi(x)$, both are no less than $\psi(x)$, yielding the second statement.
    \end{proof}

    The next result shows why the sequence of functions $(V^{(n)}_1)_{n\geq 1}$ is important.
    \begin{proposition} \label{prop: Vn}
        The following statements are true for all $n \geq 1$.
        \begin{enumerate}
            \item $\forall x \notin D^{(n)}_1 \cup S_1$, $\Q[V^{(n)}_1] (x) - \beta V^{(n)}_1(x)  = 0$
            \item $\psi \leq V^{(n)}_1 \leq V^{(n+1)}_1 \leq \phi$
            \item $V^{(n)}_1 = \psi$ on $D^{(n)}_1$, $V^{(n)}_1 = \phi$ on $S_1$ and the strict inequality $V^{(n)}_1 > \psi$ holds outside $D^{(n)}_1 \cup S_1$.
        \end{enumerate}
    \end{proposition}
    \begin{proof}
        Statement $\operatorname{(i)}$ is just a direct application of Lemma \ref{lem: Q-beta = 0}. The last inequality in $\operatorname{(ii)}$ was proved in Lemma \ref{prop: obs 1} and thanks to Lemma \ref{cor: direct of prop obs 1}, the first inequality then automatically holds true if we can prove the middle one, i.e. the sequence $(V^{(n)}_1)_{n\geq 1}$ is increasing. To this end, since $D^{(n)}_1 \cap S_1 = \emptyset $, so (recall the definitions \eqref{eq: V^1_1}, \eqref{eq: V^(n)_1}) by Lemma \ref{lem: Q-beta = 0},  we have $V^{(n)}_1= \phi$ on $S_1$ and $V^{(n)}_1= \psi$ on $D^{(n)}_1$, thus gives the first half of $\operatorname{(iii)}$. Next, we have 
            \begin{align*}
                V^{(n+1)}_1(x) &= \E_x \left[\mathrm{e}^{-\beta \tau^{(n+1)}_1} \psi(X_{\tau^{(n+1)}_1}) \mathbf{1}_{ \{\tau^{(n+1)}_1 \leq \gamma_1 \}}  + \mathrm{e}^{-\beta \gamma_1} \phi(X_{\gamma_1}) \mathbf{1}_{ \{\tau^{(n+1)}_1 > \gamma_1 \}}  \right] \\
                &=  \E_x \left[\mathrm{e}^{-\beta \tau^{(n+1)}_1} V^{(n)}_1(X_{\tau^{(n+1)}_1}) \mathbf{1}_{\{ \tau^{(n+1)}_1 \leq \gamma_1\}}  + \mathrm{e}^{-\beta \gamma_1} V^{(n)}_1(X_{\gamma_1}) \mathbf{1}_{\{ \tau^{(n+1)}_1 > \gamma_1\}}  \right] \\
                &= \E_x \left[ \mathrm{e}^{-\beta (\tau^{(n+1)}_1 \wedge \gamma_1) } V^{(n)}_1(X_{\tau^{(n+1)}_1 \wedge \gamma_1}) \right] \\
                &= V^{(n)}_1(x) + \E_x \left[ \int^{\tau^{(n+1)}_1 \wedge \gamma_1}_0 \mathrm{e}^{-\beta s} \left( \Q[V^{(n)}_1](X_s) -\beta V^{(n)}_1(X_s) \right)  ds\right] \\
                &\geq  V^{(n)}_1(x),
            \end{align*}
    where in the last inequality, we observe that
    	\begin{align}
		(D_1^{(n+1)} \cup S_1)^c =  (D_1^{(n)} \cup S_1)^c \cup \big( \{ Q[V_1^{(n)}] - \beta V_1^{(n)} >0\} \cap S_1^c\big),
	\end{align}
    so $ Q[V_1^{(n)}] - \beta V_1^{(n)}  \geq 0 $ outside $D^{(n+1)}_1 \cup S_1 $  by  $\operatorname{(i)}$ and the definition of $D^{(n+1)}_1$. This proves $(V^{(n)}_1)_{n\geq 1}$ is increasing.

    For the second half of $\operatorname{(iii)}$, we proceed by induction. For $n = 1$, Lemma \ref{cor: direct of prop obs 1} gives that $V^{(1)}_1 > \psi$ outside $D^{(1)}_1 \cup S_1$. Assume now that the claim holds for some $n \geq 1$, and we show it also holds for $n+1$. Take any $x \in D^{(n)}_1 \setminus D^{(n+1)}_1$, and define
        \begin{align*}
            \tau' \coloneqq  \inf \{t \geq 0: X_t  \notin D^{(n)}_1 \setminus D^{(n+1)}_1  \}.
        \end{align*}
    Note that $\tau^{(n+1)}_1 \wedge \gamma_1 \geq \tau' > 0$ $\mathbb{P}_x$-almost surely.  
Applying Lemma~\ref{SMP} then yields
        \begin{align*}
            V^{(n+1)}_1(x) &= \E_x \left[ \mathrm{e}^{-\beta (\tau^{(n+1)}_1 \wedge \gamma_1) } V^{(n+1)}_1(X_{\tau^{(n+1)}_1 \wedge \gamma_1}) \right] = \E_x \left[ \mathrm{e}^{-\beta \tau' } V^{(n+1)}_1(X_{\tau'}) \right] \\
            &\geq \E_x \left[ \mathrm{e}^{-\beta \tau' } V^{(n)}_1(X_{\tau'}) \right] \\
            &\geq V^{(n)}_1(x) + \E_x \left[ \int^{\tau'}_0 \mathrm{e}^{-\beta s} \left( \Q[V^{(n)}_1](X_s) -\beta V^{(n)}_1(X_s) \right)ds \right] \\
            &> V^{(n)}_1(x).
        \end{align*}
    because  $\Q[V^{(n)}_1] -\beta V^{(n)}_1 > 0$ on $D^{(n)}_1 \setminus D^{(n+1)}_1$. This shows $V^{(n+1)}_1(x) > V^{(n)}_1(x) \geq \psi(x)$ for all $x \in D^{(n)}_1 \setminus D^{(n+1)}_1$. We also have $V^{(n+1)}_1(x) \geq  V^{(n)}_1(x) > \psi(x) $ for all $x \notin  D^{(n)}_1 \cup S_1$ and thus $V^{(n+1)}_1(x) > \phi(x)$ for all $x \notin D^{(n+1)}_1 \cup S_1 $, ending the induction proof.
    \end{proof}

    Since the sequence of sets $(D^{(n)}_1)_{n\geq 1}$ is decreasing, disjoint from $S_1$ by definition and the sequence of functions $(V^{(n)}_1)_{n\geq 1}$ is increasing, we can define 
        \begin{align*}
            D_1 \coloneqq  \bigcap_{n\geq 1} D^{(n)}_1, \quad \tau_1 \coloneqq  h(D_1) \qquad \text{and} \qquad V^\infty_1 \coloneqq \lim_{n\rightarrow + \infty} V^{(n)}_1.
        \end{align*}
    \begin{theorem} \label{theo: V1 properties}
        We have $V_1 = V^\infty_1$ and for all $x \in E$, $V_1(x) = R_x(\tau_1,\gamma_1)$. Moreover, the following hold
        \begin{enumerate}
            \item $V_1 = \psi$ on $D_1$,  $V_1 = \phi$ on $S_1$ and $V_1 > \psi$ outside $D_1 \cup S_1$
            \item $\Q[V_1] -\beta V_1 =0$ outside $D_1 \cup S_1$
            \item $\Q[V_1] -\beta V_1 \leq 0$ on $D_1$
        \end{enumerate}
    \end{theorem}
    \begin{proof}
        Assume for now that $V_1 = V^\infty_1$. For $\operatorname{(i)}$, if $x \in D_1$ and $y \in S_1$ then respectively $V^{(n)}_1 (x) = \psi(x)$ and $V^{(n)}_1 (y) = \phi(y)$  for all $n \geq 1$. Hence passing to the limit $V^\infty_1 = \psi$ on $D_1$ and $V^\infty_1 = \phi$ on $S_1$. If $x \notin D_1 \cup S_1$,  then there is an $n\geq1$ such that $x \notin D^{(n)}_1 \cup S_1$, and by Proposition \ref{prop: Vn}, we have $V_1^\infty(x) \geq V_1^{(n)}(x) > \psi(x)$. 
        
        For $\operatorname{(ii)}$, if $x \notin D_1 \cup S_1$, then there exists $k$ such that $x \notin D^{(n)}_1$ for all $n \geq k$, thus
            \begin{align*}
                \Q[V^\infty_1](x) -\beta V^\infty_1(x) = \lim_{n\rightarrow + \infty} \Q[V^{(n)}_1](x) -\beta V^{(n)}_1(x) = 0.
            \end{align*}
        We prove the last statement $\operatorname{(iii)}$. If $x \in D_1$, then $x \in D^{(n)}_1$ for all $n\geq 1$, which means $Q[V^{(n)}_1](x) -\beta V^{(n)}_1(x) \leq 0 $ for all $n \geq 1$. Thus, passing to the limit we get 
            \begin{align*}
                \Q[V^\infty_1](x) -\beta V^\infty_1(x) = \lim_{n\rightarrow + \infty} \Q[V^{(n)}_1](x) -\beta V^{(n)}_1(x) \leq 0.
            \end{align*}
        Finally, we need to prove $V_1 = V^\infty_1$. From Lemma \ref{lem: Q-beta = 0}, we can conclude that $V^\infty_1(x) = R_x(\tau_1,\gamma_1)$. For any stopping time $\tau$, we have 
            \begin{align*}
                        \forall x \in E, \quad R_x(\tau, \gamma_1) &=\E_x \left[\mathrm{e}^{-\beta \tau }\psi(X_{\tau}) \mathbf{1}_{ \{\tau \leq \gamma_1 \}}  + \mathrm{e}^{-\beta \gamma_1} \phi(X_{\gamma_1}) \mathbf{1}_{ \{\tau > \gamma_1 \}}  \right] \\
                &= \E_x \left[\mathrm{e}^{-\beta \tau }\psi(X_{\tau}) \mathbf{1}_{ \{\tau \leq \gamma_1\}}  + \mathrm{e}^{-\beta \gamma_1} V^\infty_1(X_{\gamma_1}) \mathbf{1}_{ \{\tau > \gamma_1 \}}  \right] \\
                &\leq \E_x \left[\mathrm{e}^{-\beta \tau }V^\infty_1(X_{\tau}) \mathbf{1}_{ \{\tau \leq \gamma_1\}}  + \mathrm{e}^{-\beta \gamma_1} V^\infty_1(X_{\gamma_1}) \mathbf{1}_{\{ \tau > \gamma_1 \}}  \right]\\
                &= V^\infty_1(x) + \E_x \left[ \int^{\tau\wedge \gamma_1}_0 \mathrm{e}^{-\beta s} \Big( \Q[V^\infty_1](X_s) -\beta V^\infty_1(X_s) \Big) ds\right]\\
                &\leq V^\infty_1(x)
            \end{align*}
        because $\Q[V^\infty_1] -\beta V^\infty_1 \leq 0$ outside the set $S_1$. This gives 
            \begin{align*}
               \forall x \in E, \quad  V^\infty_1(x) \geq \sup_\tau R_x(\tau,\gamma_1) = V_1(x) \geq R_x(\tau_1,\gamma_1) = V^\infty_1(x),
            \end{align*}
        and so $V_1 = V^\infty_1$.
    \end{proof}
     
     Now comes the important step in our algorithm. We have so far  defined $V_1$ that lies between $\psi$ and $\phi$. Suppose for now that the triple $(D_k,S_k,V_k)$ has been obtained for some $k \geq 1$, we define the triple $(D_{k+1},S_{k+1}, V_{k+1})$ inductively as follows. Set 
        \begin{align} \label{S_k}
            S_{k+1} = \Big(S_k \cap \{ \Q[V_k] - \beta V_k \geq 0 \} \Big) \cup \{ \phi = \psi\}, \qquad \gamma_{k+1} \coloneqq  h(S_{k+1}),
        \end{align}
and define 
	\begin{align*}
		\forall x \in E, \quad	V_{k+1}(x) \coloneqq  \sup_{\tau}R_x(\tau, \gamma_{k+1}).
	\end{align*}
        The set $D_{k+1}$ is then obtained from the same procedure above for $V_1$ such that
            \begin{align} \label{V_k chracter}
                \forall x \in E, \quad V_{k+1}(x) =R_x(\tau_{k+1}, \gamma_{k+1}), \qquad \text{where} \quad \tau_{k+1} = h(D_{k+1}).
            \end{align}
        For example, let us illustrate the second step $k=2$ to compute the function $V_2$. We just proceed analogously for the function $V_1$. Initially, we set $S_2 \coloneqq  (S_1 \cap \{ \Q[V_1] - \beta V_1 \geq 0 \}) \cup \{ \phi = \psi \}$, and define
            \begin{align*}
                D^{(1)}_2 \coloneqq  \{ \Q[\psi] - \beta \psi \leq 0\} \cap S_2^c, \qquad  \tau^{(1)}_2 \coloneqq  h(D^{(1)}_2), \qquad V^{(1)}_2(x)\coloneqq  R_x(\tau^{(1)}_2, \gamma_2).
            \end{align*}
        Next, suppose we have defined $(D^{(n)}_2, \tau^{(n)}_2 , V^{(n)}_2)$ for some $n \geq 1$, we define the next triple $(D^{(n+1)}_2, \tau^{(n+1)}_2 , V^{(n+1)}_2 )$ by
            \begin{align*}
                D^{(n+1)}_2\coloneqq  \{ \Q[V^{(n)}_2] - \beta V^{(n)}_2  \leq 0 \} \cap D^{(n)}_2, \qquad  \tau^{(n+1)}_2 = h(D^{(n+1)}_2),
            \end{align*}
and 
	\begin{align*}
		 \forall x \in E, \quad V^{(n+1)}_2(x)\coloneqq  R_x(\tau^{(n+1)}_2, \gamma_2).
	\end{align*}
	
        Finally, the set $D_2$ is obtained as the limit $D_2 \coloneqq  \bigcap_{n \geq 1} D^{(n)}_2$. Adapting the same proof of Theorem \ref{theo: V1 properties}, we can show that $V_2(x) = R_x(\tau_2, \gamma_2) = \lim_{n\rightarrow + \infty} V^{(n)}_2(x)$ and every statement in Theorem \ref{theo: V1 properties} also holds true for $V_2$ (with all the subscript 1's replaced by 2). Repeat the same procedure for $k = 3,4,...$ to obtain the sequence $(D_k, S_k, V_k)_{k \geq 1}$. For the sake of completeness we state the following general properties of the sequence $(V_k)_{k \geq 1}$. 
        \begin{theorem} \label{theo: big seq}
            The following statements hold for all $k \geq 1$
            \begin{enumerate}
                \item $\forall x \in E$, $V_k(x) = R_x(\tau_k, \gamma_k)$
                \item $V_k = \psi$ on $D_k$, $V_k = \phi$ on $S_k$ and  $V_k > \psi$ outside $D_k \cup S_k$ (note that $D_k \cap S_k = \emptyset$).
                \item $\Q[V_k] -\beta V_k \leq 0 $ on $D_k$ and $\Q[V_k] -\beta V_k = 0$ outside $D_k \cup S_k$
                \item $\psi \leq V_{k+1} \leq V_k \leq \phi$
                \item $(D_k)_{k \geq 1}$ is increasing and $(S_k)_{k \geq 1}$  is decreasing
            \end{enumerate}
        \end{theorem}
        \begin{proof}
        The first three statements together with the inequality $\psi \leq V_k$ for any $k \geq 1$ can be obtained easily by adapting the same proofs of Proposition \ref{prop: Vn} and Theorem \ref{theo: V1 properties} above (with virtually no changes except the subscripts). In $\operatorname{(v)}$, we already have $(S_k)_{k \geq 1}$ is decreasing by its definition in \eqref{S_k}, thus establishing half of it. Now, we prove $V_{k+1} \leq V_k$ in $\operatorname{(iv)}$, and when this is established, all of the inequalities $V_{k+1} \leq V_k \leq V_1 \leq V_0 \wedge \phi \leq  \phi$ immediately follow from Lemma \ref{cor: direct of prop obs 1}. Indeed, $S_{k+1} \subset S_k$ implies $V_{k+1} = V_k = \phi$ on $S_{k+1}$ by Lemma \ref{lem: Q-beta = 0} (recall that $S_{k+1} = \Big(S_k \cap \{ \Q[V_k] - \beta V_k \geq 0 \} \Big) \cup \{ \phi = \psi\}$). Item $\operatorname{(iii)}$ implies that $\Q[V_k] - \beta V_k \leq 0$ outside $S_{k+1}$. Therefore, for all $x \in E$,
                \begin{align*}
                    V_{k+1}(x) = R_x(\tau_{k+1}, \gamma_{k+1}) &\leq E_x[\mathrm{e}^{-\beta (\tau_{k+1} \wedge \gamma_{k+1})}V_{k}(X_{\tau_{k+1} \wedge \gamma_{k+1}})] \\
                    &= V_k(x) + \E_x \left[ \int^{\tau_{k+1} \wedge \gamma_{k+1}} _{0} \mathrm{e}^{-\beta s}\Big( \Q[V_k](X_s) - \beta V_k(X_s) \Big) ds \right] \\
                    &\leq V_k(x) 
                \end{align*}
        because before $ \tau_{k+1} \wedge \gamma_{k+1} = h(D_{k+1} \cup S_{k+1})$ the process is still outside $S_{k+1}$, implying $ \Q[V_k](X_s) - \beta V_k(X_s) \leq 0$.
        Only the first half of  $\operatorname{(v)}$ remains. Let $x \in D_k$ (which is disjoint from $S_k \supset \{\phi = \psi \}$), we have $\phi(x) > V_k(x) = \psi(x)$ and this implies $\phi(x) > V_{k+1}(x) = \psi(x)$ because $V_{k+1}(x) \leq V_k(x)$. It follows that $x \in D_{k+1}$ because if $x \notin D_{k+1}$ then either $V_{k+1}(x) > \psi(x)$ (if $x \notin D_k \cup S_k$) or $V_{k+1}(x) = \phi(x)$ (if $x \in S_k$), a contradiction.
        \end{proof}

        Theorem \ref{theo: big seq} allows us to define the triple $(D_\infty, S_\infty, V_\infty)$ as follows 
            \begin{align*}
                D_\infty = \bigcup_{k \geq 1} D_k, \qquad S_\infty \coloneqq  \bigcap_{k \geq 1} S_k, \qquad V_\infty(x) = \lim_{k \rightarrow + \infty}  V_k(x) = \inf_{k\geq 1} V_k(x),
            \end{align*}
        and let
            \begin{align*}
                \tau_{\psi,\phi} \coloneqq  h(\{ \phi = \psi \}), \qquad \tau_{\infty} \coloneqq  h(D_\infty), \qquad \gamma_\infty \coloneqq  h(S_\infty).
            \end{align*}
\begin{rem}
        		We observe from the definition of the sequence $(D_k,S_k,V_k)_{k \geq 0}$ in \eqref{S_k} that if there is an integer $N \geq 1$ such that $S_{N+1} = S_N$, then necessarily $V_{N+1} = V_N$ and $D_{N+1} = D_N$. Consequently, the sequence stabilizes:
		\begin{align}
			\forall k \in \mathbb N, \quad S_{N+k} = S_N = S_\infty,\quad  V_{N+k} = V_N = V_\infty, \quad D_{N+k} = D_N= D_\infty	
		\end{align}	
		 This means that the sequence $(D_k,S_k,V_k)_{k \geq 0}$ becomes stationary after a finite number of iterations.. If $N$ is the smallest integer for which $S_{N+1} = S_N$, then we have the bound (recalling $\{ \phi =\psi \} \subset S_k$ for all $k \geq 1$)
			\begin{align}
				N \leq \min\{ |E| - | \{ \phi = \psi \} |,  |S_1|  \} = \min\{ |E| - | \{ \phi = \psi \} |, | \{ V_0 > \phi \} \cup \{ \phi = \psi \} | \},
			\end{align}
	where $| \cdot |$ denotes cardinality. In particular, we have $N < |E|$,  meaning that the number of iterations required to reach $V_\infty$ is strictly less than the number of the states in $E$.
        		\end{rem}
		
        The next result is the main result of this paper
        \begin{theorem} \label{theo: main}
            The pair $(\tau_\infty \wedge \tau_{\psi,\phi}, \gamma_\infty)$ forms a NE and $V_\infty = V$. (Note that $\tau_\infty \wedge \tau_{\psi,\phi}$  is the stopping time  $h(\{ \phi = \psi \} \cup D_\infty)$.) Additionally, we have  $V_\infty = \phi$ on $S_\infty$ (this does not mean $S_\infty = \{ V = \phi \}$) and $V_\infty = \psi < \phi$ on $D_\infty$.
        \end{theorem}
        \begin{proof}
            From the previous remark, there is $N < |E|$ such that  $S_{N+k} = S_N = S_\infty$ for all $k \geq 1$. Hence, we have
                \begin{align*}
                    \Big(S_N\cap \{ \phi > \psi \}  \cap \{ \Q[V_N] - \beta V_N \geq 0 \} \Big) \cup \{ \phi = \psi \} = S_N,
                \end{align*}
            which implies $S_N \cap \{ \phi > \psi \} \subset \{ \Q[V_N] - \beta V_N \geq 0 \}$. We also have $D_{N+k} = D_N$ for all $k \geq 1$, and by Theorem \ref{theo: big seq} $\operatorname{(iii)}$, we get $\Q[V_N] - \beta [V_N] = 0$ outside $D_N \cup S_N$ and $\Q[V_N] - \beta [V_N] \leq 0$ on $D_N$. Applying Lemma \ref{lem: critical obs} and Theorem \ref{theo: sufficient for NE} for
            \begin{align*}
                A= D_\infty = D_N, \qquad  B = S_{N}\cap \{ \phi > \psi \} = S_{\infty}\cap \{ \phi > \psi \},
            \end{align*}
            we get $V_\infty(x) = R_x(\tau_\infty, \gamma_\infty) = R_x(\tau_\infty \wedge \tau_{\psi,\phi}, \gamma_\infty)$ and the pair $(\tau_\infty \wedge \tau_{\psi,\phi}, \gamma_\infty)$ is also a NE. Thus, $V_\infty =V$, which is the desired result. The last statement follows from the expression of $V_\infty(x) = R_x(\tau_\infty \wedge \tau_{\psi,\phi}, \gamma_\infty)$.
        \end{proof}
        As a consequence of Theorems \ref{theo: big seq} and \ref{theo: main}, the following result is obtained by taking $A = D_\infty$ and $B = S_\infty \cap \{ \phi > \psi \}$.
        \begin{cor}[Necessity part for Theorem \ref{theo: sufficient for NE}] \label{cor: necessity part}
        If $V$ is the value function then there exist  disjoint sets $A$ and  $B$ such that $A \cap \{ \phi = \psi \} = B \cap \{ \phi = \psi \} = \emptyset$ satisfying \eqref{conditions for a NE}.
        \end{cor}
    \subsection{Auxillary results and discussions} \label{subsec: aux}
    \subsubsection{On the choice of the set $S_1 = \{ V_0 > \phi \} \cup \{ \phi = \psi \}$.}
        In the above algorithm, we may have $S_\infty \neq  \{ V = \phi \}$, but it is always true that $D_\infty \cup \{ \phi = \psi\} = \{ V = \psi \}$. If one wants to recover the set $\{ V = \phi \}$, the set $S_1$ in \eqref{def S_1} should be replaced by the set $\widetilde S_1 \coloneqq  \{ V_0 \geq \phi \}$ and then follows the steps in the previous section to obtain a new decreasing sequence $(\widetilde V_k)_{k\geq 1}$. The proofs of previous results remain unchanged. Observe that $S_1 \subset \widetilde S_1$ and in fact $\widetilde S_1 \setminus S_1 = \{ V_0 = \phi \} \cap \{ \phi > \psi \}$. In particular, if $\{ V_0 =\phi \} = \emptyset$ then  $S_1 = \widetilde S_1$, and the procedure yields the same sequence $(V_k)_{k \geq 1}$. More precisely, we have 
            \begin{proposition} \label{S' to initialize}
                Suppose $(\widetilde D_k, \widetilde \tau_k, \widetilde S_k, \widetilde \gamma_k, \widetilde V_k)_{k \in \mathbb N \cup \{\infty\}}$ be the sequence of quintuplet obtained from re-performing all the steps in previous section with $S_1$ replaced by $\widetilde S_1 =\{ V_0 \geq \phi \} \supset S_1$. Then all statements in Theorems \ref{theo: big seq} and \ref{theo: main} hold for $(\widetilde D_k, \widetilde \tau_k, \widetilde S_k, \widetilde \gamma_k, \widetilde V_k)_{k \in \mathbb N \cup \{\infty\}}$. Moreover, we have
                    \begin{enumerate}
                        \item $\widetilde D_\infty = D_\infty $, and $\widetilde D_\infty \cup \{ \phi = \psi \} = \{ V = \psi \}$. In particular, $\widetilde \tau_\infty = h(\widetilde D_\infty) = h(D_\infty) = \tau_\infty$.
                        \item $\widetilde V_\infty = V_\infty = V$
                        \item $\widetilde V_k < \phi$ outside $ \widetilde S_k$ for all $ k \in \mathbb N \cup \{ \infty \}$
                        \item $\psi < \widetilde V_k < \phi$ outside $\widetilde D_k \cup \widetilde S_k$ for all $ k \in \mathbb N \cup \{ \infty \}$
                        \item $\widetilde S_\infty = \{ V = \phi \}$
                    \end{enumerate}
            \end{proposition}
            \begin{proof}
                Items $\operatorname{(i)}, \operatorname{(ii)}$  are straightforward from construction and from the proofs in the previous subsection so we only prove the last three by induction. Consider the function $\widetilde V_1$ and we know that  $\widetilde V_1\leq V_0 \wedge \phi$. Now $\widetilde V_1 = \phi$ on $\widetilde S_1$ and $\widetilde V_1 \leq V_0 < \phi$ outside $\widetilde S_1$, thus $\operatorname{(iii)}$ is true for $k = 1$. Suppose this is true for some $k \geq 1$, we prove it is also true for $k+1$. Indeed, set 
                    \begin{align*}
                        \Delta \coloneqq  \inf\{t \geq 0: X_t \notin \widetilde S_{k} \setminus \widetilde S_{k+1} \},
                    \end{align*}
                and observe that for all $x \in \widetilde S_{k} \setminus \widetilde S_{k+1} $, 
                \begin{align}
                		\widetilde \tau_{k+1} \wedge \widetilde \gamma_{k+1} = h(\widetilde D_{k+1} \cup \widetilde S_{k+1} ) \geq  h\Big( (\widetilde S_k)^c \cup \widetilde S_{k+1}\Big) = \Delta >0 \quad \mathbb P_x-\text{a.s.}
                \end{align}
       because $\widetilde D_{k+1}  \subset \widetilde D_k \subset (\widetilde S_k)^c$ (recall that $\widetilde D_k$ and $\widetilde S_k$ are disjoint). Using Lemma \ref{SMP} and the fact that $\widetilde V_{k+1} \leq \widetilde V_k$, we have for all $x \in \widetilde S_{k} \setminus \widetilde S_{k+1} $,
                    \begin{align*}
                        \widetilde V_{k+1}(x) = R_x(\widetilde \tau_{k+1}, \widetilde \gamma_{k+1}) &= \E_x[\mathrm{e}^{-\beta (\widetilde \tau_{k+1} \wedge \widetilde \gamma_{k+1})}\widetilde V_{k+1}(X_{\widetilde \tau_{k+1} \wedge \widetilde \gamma_{k+1}})] \\
                        &= \E_x[\mathrm{e}^{-\beta \Delta }\widetilde V_{k+1}(X_{\Delta})]  \\
                        &\leq \E_x[\mathrm{e}^{-\beta \Delta }\widetilde V_k(X_{\Delta})] \\
                        &= V_k(x) + \E_x \left[ \int^{\Delta} _{0} \mathrm{e}^{-\beta s}\Big( \Q[V_k](X_s) - \beta V_k(X_s) \Big) ds \right] \\
                        &< V_k(x) \\
                        &\leq \phi(x). 
                    \end{align*}
                because $ \Q[V_k] - \beta V_k < 0$ on $ \widetilde S_{k} \setminus \widetilde S_{k+1}$. But $ \widetilde V_{k+1} \leq  \widetilde V_k < \phi$ outside $\widetilde S_k$ and thus $\widetilde V_{k+1} < \phi$ outside $\widetilde S_{k+1}$. Passing this to the limit we get $\widetilde V_\infty < \phi$ outside $\widetilde S_\infty$. Now $\operatorname{(iv)}$ follows from $\operatorname{(iii)}$ and an analogue of Theorem \ref{theo: big seq} $\operatorname{(ii)}$. 
                Finally for $\operatorname{(v)}$, since  $\widetilde V_\infty > \psi $ outside $\widetilde D_\infty \cup \widetilde S_\infty$, it is not hard to see that $\widetilde S_\infty = \{ V = \phi \}$.
            \end{proof}
	
	Although it is rare in practice to encounter the case $\{ V_0 \geq \phi \} \neq \{ V_0 > \phi \}$,  the inclusion $S_1 \subset \widetilde S_1$ guarantees that $S_k \subset \widetilde S_k$ for all $k \in \mathbb N$ (see Theorem \ref{compare Sk and tilde Sk} below). As a result, initializing the algorithm with  $S_1$ is expected to lead to faster convergence compared to initialization with $\widetilde S_1$. For example, consider the case where $\phi > \psi$ and $|S_1| = 1$ while $|\widetilde S_1 |$  is 100 or more. In this case, initializing with $S_1$ would require only one iteration (as $S_2$ would be either empty or equal to $S_1$) making it evidently more efficient. Moreover, the two approaches may lead to different sets $S_\infty$ and $\widetilde S_\infty$ when $S_1 \neq \widetilde S_1$, potentially resulting in multiple NE. Examples illustrating both $S_\infty = \widetilde S_\infty$ and $S_\infty \neq \widetilde S_\infty$ are given in Section \ref{differentlimiting}. If our primary goal is to compute the value function $V$, initializing with $S_1$ consistently yields a faster and more efficient algorithm. If our goal is to recover the critical sets $\{ V =\psi \}$ and $\{ V = \phi \}$, then initializing with $\widetilde S_1$ is recommended. Alternatively, we can initialize with $S_1$, compute $V$ and then check which $x \in E$ such that $V(x) = \phi(x)$.
         \begin{theorem}\label{compare Sk and tilde Sk}
         	We have $S_k \subset \widetilde S_k$, $\widetilde D_k \subset D_k$ and $\widetilde V_k \geq V_k$ for all $k \in \mathbb N$.
         \end{theorem}
         
         The proof of of Theorem \ref{compare Sk and tilde Sk} is a consequence of the following lemmas.
         \begin{lemma} \label{V1V1}
         	We have $\widetilde D^{(n)}_1 \subset  D^{(n)}_1$ and that $V^{(n)}_1 \leq \widetilde V^{(n)}_1$ for all $n \in \mathbb N$. As a consequence $V_1 \leq \widetilde V_1$ and $\widetilde D_1 \subset D_1$.
         \end{lemma}
         \begin{proof} 
         	For $n=1$, recall that $\widetilde \tau^{(1)}_1 = h(\widetilde D^{(1)}_1) $, $\tau^{(1)}_1 = h(D^{(1)}_1) $,  $\gamma_1 = h(S_1)$ and $\widetilde \gamma_1 =  h(\widetilde S_1)$, where 
			\[ \widetilde D^{(1)}_1  = \{\Q[\psi] -\beta \psi \leq 0 \} \setminus \widetilde S_1  \subset \{\Q[\psi] -\beta \psi \leq 0 \} \setminus  S_1  =   D^{(1)}_1,\]
		and that 
			\[ \forall x \in E, \quad V^{(1)}_1(x) = R_x (\tau^{(1)}_1, \gamma_1), \quad \widetilde V^{(1)}_1(x) = R_x (\widetilde \tau^{(1)}_1, \widetilde \gamma_1). \]
		We are going to prove that $\widetilde V^{(1)}_1 \geq V^{(1)}_1$. Indeed, for  $y \in \widetilde S_1$, we have $\widetilde V^{(1)}_1(y) = \phi(y) \geq V^{(1)}_1(y)$ while for $ x \in \widetilde D^{(1)}_1$, we have $V^{(1)}_1(x) = \psi(x) = \widetilde V^{(1)}_1(x)$. It remains to prove $\widetilde V^{(1)}_1  \geq V^{(1)}_1$ outside the set $\widetilde D^{(1)}_1 \cup \widetilde S_1 = D^{(1)}_1 \cup \widetilde S_1$ (and so $ \widetilde \tau^{(1)}_1 \wedge \widetilde \gamma_1 = \tau^{(1)}_1 \wedge \widetilde \gamma_1 $). Consider one $x \notin D^{(1)}_1 \cup \widetilde S_1$. Then, applying Lemma \ref{SMP} yields
		\begin{align*}
		\widetilde V^{(1)}_1(x) &= \E_x[ \mathrm{e}^{-\beta ( \widetilde \tau^{(1)}_1 \wedge \widetilde \gamma_1) } \widetilde V^{(1)}_1(X_{\widetilde \tau^{(1)}_1 \wedge \widetilde \gamma_1}) ] \\
		&= \E_x[ \mathrm{e}^{-\beta (  \tau^{(1)}_1 \wedge \widetilde \gamma_1) } \widetilde V^{(1)}_1(X_{ \tau^{(1)}_1 \wedge \widetilde \gamma_1}) ] \\
		&\geq \E_x[ \mathrm{e}^{-\beta ( \tau^{(1)}_1 \wedge \widetilde \gamma_1) } V^{(1)}_1(X_{\tau^{(1)}_1 \wedge \widetilde \gamma_1}) ] \quad ( \text{since} \ \widetilde V^{(1)}_1 \geq  V^{(1)}_1 \ \text{on} \ D^{(1)}_1 \cup \widetilde S_1 )\\
		&=  V^{(1)}_1(x) +  \E_x \left[ \int^{\tau^{(1)}_1 \wedge \widetilde \gamma_1}_0 \mathrm{e}^{-\beta u} \Big( \Q[V^{(1)}_1](X_u) - \beta V^{(1)}_1 (X_u)\Big)du \right]\\
		&= V^{(1)}_1(x),
		\end{align*}
	because $ \Q[V^{(1)}_1] - \beta V^{(1)}_1 = 0$ outside $ D^{(1)}_1 \cup  S_1 \subset D^{(1)}_1 \cup \widetilde S_1 $. This ends the proof that $\widetilde V^{(1)}_1 \geq  V^{(1)}_1$. \\
	\indent Suppose that the claim of the lemma is true for $n$, we prove it is true for $n+1$. First, recall 
		\[  \widetilde D^{(n+1)}_1 = \{ \Q[\widetilde V^{(n)}_1] - \beta \widetilde V^{(n)}_1 \leq 0 \} \cap \widetilde D^{(n)}_1, \quad  D^{(n+1)}_1 = \{ \Q[V^{(n)}_1] - \beta  V^{(n)}_1 \leq 0 \} \cap  D^{(n)}_1. \]
		Let $x \in \widetilde D^{(n+1)}_1$. Then $x \in \widetilde D^{(n)}_1 \subset D^{(n)}_1 $ and consequently $\widetilde V^{(n)}_1(x) = \psi(x) =  V^{(n)}_1(x)$. We also have 
		\begin{align*}
			0& \geq  \Q[ \widetilde V^{(n)}_1](x) -\beta \widetilde V^{(n)}_1(x)  \\
			&= \sum_{y \neq x} Q(x,y)\widetilde V^{(n)}_1(y) + (Q(x,x) -\beta)  \widetilde V^{(n)}_1(x)  \\
			&=  \sum_{y \neq x} Q(x,y)\widetilde V^{(n)}_1(y) + (Q(x,x) -\beta) \psi(x) \\
			&\geq  \sum_{y \neq x} Q(x,y)V^{(n)}_1(y) + (Q(x,x) -\beta) V^{(n)}_1(x) \\
			&= \Q[ V^{(n)}_1](x) -\beta  V^{(n)}_1(x),
		\end{align*}
		which shows $x \in \{ \Q[V^{(n)}_1] - \beta  V^{(n)}_1 \leq 0 \}$, and hence $\widetilde D^{(n+1)}_1 \subset D^{(n+1)}_1 $. The proof of $\widetilde V_1^{(n+1)} \geq V_1^{(n+1)}$ follows similarly as in the case $n=1$. First, consider $x \in D^{(n+1)}_1$, then $\widetilde V_1^{(n+1)} (x) \geq \psi(x) = V_1^{(n+1)}(x) $ (the proof of $\widetilde V_1^{(n)}  \geq \psi$ follows similarly in  Lemma \ref{cor: direct of prop obs 1}  and Proposition \ref{prop: Vn}). If $x \in \widetilde S_1$, then $ \widetilde V_1^{(n+1)} (x) = \phi(x) \geq V_1^{(n+1)} (x)$. It remains to show that $\widetilde V_1^{(n+1)} \geq V_1^{(n+1)}$ outside $D^{(n+1)}_1 \cup \widetilde S_1$. Consider one $x \notin  D^{(n+1)}_1 \cup \widetilde S_1$, we have
			\[ \widetilde \tau^{(n+1)}_1 \wedge \widetilde \gamma_1 \geq  \tau^{(n+1)}_1 \wedge \widetilde \gamma_1  \quad \mathbb P_x-\text{a.s.}\]
Applying Lemma \ref{SMP} yields
		\begin{align*}
		\widetilde V^{(n+1)}_1(x) &= \E_x[ \mathrm{e}^{-\beta ( \widetilde \tau^{(n+1)}_1 \wedge \widetilde \gamma_1) } \widetilde V^{(n+1)}_1(X_{\widetilde \tau^{(n+1)}_1 \wedge \widetilde \gamma_1}) ] \\
		&= \E_x[ \mathrm{e}^{-\beta (  \tau^{(n+1)}_1 \wedge \widetilde \gamma_1) } \widetilde V^{(n+1)}_1(X_{ \tau^{(n+1)}_1 \wedge \widetilde \gamma_1}) ] \\
		&\geq \E_x[ \mathrm{e}^{-\beta ( \tau^{(n+1)}_1 \wedge \widetilde \gamma_1) } V^{(n+1)}_1(X_{\tau^{(n+1)}_1 \wedge \widetilde \gamma_1}) ] \quad ( \text{since} \ \widetilde V^{(n+1)}_1 \geq  V^{(n+1)}_1 \ \text{on} \ D^{(n+1)}_1 \cup \widetilde S_1 )\\
		&=  V^{(n+1)}_1(x) +  \E_x \left[ \int^{\tau^{(n+1)}_1 \wedge \widetilde \gamma_1}_0 \mathrm{e}^{-\beta u} \Big( \Q[V^{(n+1)}_1](X_u) - \beta V^{(n+1)}_1 (X_u)\Big)du \right]\\
		&= V^{(n+1)}_1(x),
		\end{align*}
because $ \Q[V^{(n+1)}_1] - \beta V^{(n+1)}_1 = 0$ outside $ D^{(n+1)}_1 \cup  S_1 \subset D^{(n+1)}_1 \cup \widetilde S_1 $. This establishes that $\widetilde V_1^{(n+1)} \geq V_1^{(n+1)}$, thereby completing the induction step.

		Therefore, we have shown that $\widetilde V_1^{(n)} \geq V_1^{(n)}$ for all $n \in \mathbb N$. Letting $n \rightarrow + \infty$, we conclude that $\widetilde V_1 \geq V_1$ and $\widetilde D_1 \subset D_1$. 
		         \end{proof}
         \begin{lemma}
         	We have $S_2 \subset \widetilde  S_2$.
         \end{lemma}
         \begin{proof}
         	We first recall that 
		\begin{align}
			S_2 = \{\Q[V_1] -\beta V_1 \geq 0 \} \cap S_1, \quad \widetilde S_2 = \{\Q[\widetilde V_1] -\beta \widetilde V_1 \geq 0 \} \cap \widetilde S_1
		\end{align}
		Let $x \in S_2$. Then $x \in S_1 \subset \widetilde S_1$ and hence $V_1(x) = \phi(x) = \widetilde V_1(x)$. Also, we have 
			\begin{align*}
				0 &\leq \Q[V_1](x) -\beta V_1(x) = \sum_{y \neq x } Q(x,y)V_1(y) +(Q(x,x)-\beta) V_1(x)\\
				&= \sum_{y \neq x } Q(x,y)V_1(y) +(Q(x,x)-\beta) \phi(x) \\
				&\leq \sum_{y \neq x } Q(x,y)\widetilde V_1(y) +(Q(x,x)-\beta) \widetilde V_1(x) \quad (\text{since} \  V_1 \leq  \widetilde V_1 \  \text{in Lemma \ref{V1V1}}) \\
				&= \Q[\widetilde V_1](x) -\beta \widetilde V_1(x).
			\end{align*}
		Therefore, $x \in \{ \Q[\widetilde V_1] -\beta \widetilde V_1 \geq 0\}$ and the claim follows. 
         \end{proof}
         \begin{proof}[Proof of Theorem \ref{compare Sk and tilde Sk}]
         The proof now follows by induction, using the same ideas and techniques as those employed in the preceding lemmas: first show that $\widetilde D_2 \subset D_2$ and $\widetilde V_2 \geq V_2$, then proceed similarly for $k =3,4, \dots$.
         \end{proof}

    \subsubsection{A way to create examples that have many NE in the case $\{ V_0 > \phi \} \neq \emptyset$.}
    
       In Proposition \ref{S' to initialize}, we proved that $\widetilde V_\infty = V_\infty = V$, so that
        \begin{align} \label{forquote}
            \forall x \in E,  \quad   R_x(\tau_\infty \wedge \tau_{\phi,\psi},  \gamma_\infty)  = V_\infty (x) = \widetilde V_\infty(x) = R_x(\tau_\infty \wedge \tau_{\phi,\psi},  \widetilde \gamma_\infty),
        \end{align}
        and that the two pairs $ (\tau_\infty \wedge \tau_{\psi,\phi}, \gamma_\infty)$ and $(\tau_\infty \wedge \tau_{\psi,\phi}, \widetilde \gamma_\infty)$ are all NE. Note that $\widetilde \gamma_\infty \leq \gamma_\infty$ because $S_\infty \subset \widetilde S_\infty$ and the two stopping times are different when the inclusion is strict. Although we have examples where the inclusion $S_\infty \subset \widetilde S_\infty$ is strict, thus yielding multiple NE, one may still wonder whether there is a systematic way to construct other stopping times $\gamma$ that preserve the property that $(\tau_\infty \wedge \tau_{\psi,\phi}, \gamma)$ is a NE. In this subsection, we aim to address this question. First, we provide an intuitive justification for the possibility of multiple NE, without invoking Proposition \ref{S' to initialize}. Next, we offer a constructive approach for demonstrating examples where more than one NE arises. 
        
        Let us begin with the first objective. For notational convenience, we set $D \coloneqq  D_\infty$, and $S \coloneqq  S_\infty \cap \{ \phi > \psi \}$, and note that $D \cap S = \emptyset$. We get
            \begin{align*}
                \forall x \in E, \quad V_\infty(x) = R_x(\tau_\infty \wedge \tau_{\phi,\psi}, \gamma_\infty) = R_x\Big(h(D \cup \{ \phi = \psi\}),  h(S \cup \{ \phi = \psi\}) \Big).
            \end{align*}
        From Theorem \ref{theo: main}, we know $V_\infty \leq \phi$, with $V_\infty = \phi$ on $S_\infty = S \cup \{ \phi = \psi \}$ and $V = \psi < \phi$ on $D$. However, outside the set $D \cup  S \cup \{ \phi = \psi\}$, no further information is available to determine whether the inequality $V_\infty \leq \phi$ is strict. Thus, perhaps there exists a non-empty set $O$, disjoint from $ D \cup  S \cup \{ \phi = \psi\} $, such that $V_\infty = \phi$ on $O$. If this is the case, we claim that the following function
            \begin{align*}
                \forall x \in E, \quad V_O(x) \coloneqq  R_x\Big(h(D \cup \{ \phi = \psi\}),  h(O \cup S \cup \{ \phi = \psi\}) \Big),
            \end{align*}
        coincides with $V_\infty$. Indeed, let $\Delta_1 \coloneqq  h(D \cup \{ \phi = \psi\})$ and $ \Delta_2 \coloneqq  h(O \cup S \cup \{ \phi = \psi\})$, then $\forall x \in E$,
            \begin{align*}
                V_O(x)  &= \E_x \left[\mathrm{e}^{-\beta \Delta_1} \psi(X_{\Delta_1}) \mathbf{1}_{\{\Delta_1 \leq \Delta_2\}} + \mathrm{e}^{-\beta \Delta_2} \phi(X_{\Delta_2}) \mathbf{1}_{\{\Delta_1 > \Delta_2\}}  \right] \\
                        &= \E_x \left[\mathrm{e}^{-\beta \Delta_1} V_\infty (X_{\Delta_1}) \mathbf{1}_{\{\Delta_1 \leq \Delta_2\}} + \mathrm{e}^{-\beta \Delta_2} V_\infty(X_{\Delta_2}) \mathbf{1}_{\{\Delta_1 > \Delta_2\}}  \right] \\
                        &= \E_x \left[\mathrm{e}^{-\beta (\Delta_1 \wedge \Delta_2)} V_\infty(X_{\Delta_1 \wedge \Delta_2}) \right] \\
                        &= V_\infty(x) + \E_x\left[ \int^{\Delta_1 \wedge \Delta_2}_0 \mathrm{e}^{-\beta s} \Big( \Q[V_\infty](X_s) - \beta V_\infty(X_s)\Big) ds  \right] \\
                        &= V_\infty(x) \label{fforcite} \numberthis
            \end{align*}
        because $ \Q[V_\infty] - \beta V_\infty = 0$ outside $D \cup S \cup O \cup \{ \phi = \psi \}$. Next, we claim that any set $O$ disjoint from $D \cup S \cup \{ \phi = \psi \}$ and satisfying $V_\infty(x) = \phi$ for all $x \in O$ gives rise to a NE pair 
        \[\Big(h(D \cup \{ \phi = \psi\}),  h(O \cup S \cup \{ \phi = \psi\}) \Big).\]
        
        At this point, we emphasize that a pair of stopping times $(\tau,\gamma)$ may yield the correct the value function, i.e. $V(x) = R_x(\tau,\gamma)$, without necessarily forming a NE. To verify that the pair $\Big(h(D \cup \{ \phi = \psi\}),  h(O \cup S \cup \{ \phi = \psi\}) \Big)$ does in fact constitute a NE, we observe
            \begin{align*}
                \forall \gamma, \qquad  R_x(h(D \cup \{ \phi = \psi\}),  \gamma ) \geq V_\infty(x) = V_O(x) = R_x\Big(h(D \cup \{ \phi = \psi\}), h(O \cup S \cup \{ \phi = \psi\}) \Big),
            \end{align*}
        and from $\psi \leq V_\infty$, $V_\infty = \phi$ on $O \cup S \cup \{ \phi = \psi \}$, we use Dynkin's formula to get
            \begin{align*}
                \forall \tau, \qquad  R_x(\tau,  O \cup S \cup \{ \phi = \psi\} ) &\leq V_\infty(x) + \E_x \left[\int^{\tau \wedge h(O\cup S \cup \{ \phi = \psi\})}_0 \mathrm{e}^{-\beta s}\Big(\Q[V_\infty](X_s) - \beta V_\infty (X_s) \Big)ds \right] \\
                &\leq V_\infty(x) = V_O(x),
            \end{align*}
        which shows the claim.
        
        Now suppose that $ |  \{ V =\phi \} \setminus  S_\infty | \geq 1$. (Note that $ \{ V =\phi \}  = \widetilde S_\infty$, but here we are intentionally disregarding the construction of $\widetilde S_\infty$ and Proposition \ref{S' to initialize}). If we take any non empty subset $O \subset\{ V =\phi \}  \setminus S_\infty$, then, from the previous analysis, we get $V_O=  V_\infty$ and the pair
        \[\Big(h(D \cup \{ \phi = \psi\}),  h(O \cup S \cup \{ \phi = \psi\}) \Big) \]
constitutes a NE. In particular, this demonstrates that there are multiple Nash equilibria in this case, thereby disproving the intuitive expectation that the NE might be unique. 

	Returning briefly to Proposition \ref{S' to initialize}, suppose that 	
\[ | \{ V =\phi \} \setminus  S_\infty | = |\widetilde S_\infty \setminus  S_\infty | \geq 2.\]
Then  we can, in fact,  choose $O$ such that $h(O \cup S \cup \{ \phi = \psi\}) $ differs from both $\gamma_\infty$ and $\widetilde \gamma_\infty$. Indeed, if $ | \widetilde S_\infty \setminus  S_\infty | \leq 1$, then any subset $O \subset \widetilde S_\infty \setminus  S_\infty$ is either empty or equal to $ \widetilde S_\infty \setminus  S_\infty$, in which case the resulting stopping time coincides with either $\gamma_\infty$ or $\widetilde \gamma_\infty$, respectively.
        
	The preceding discussion offers only an intuitive justification for the possibility of multiple NE without appealing to Proposition \ref{S' to initialize}, as the existence of a set $O$ satisfying the necessary conditions has remained a mystery until now. In general, it may happen that $S_\infty = \widetilde S_\infty$ even though $S_1 \neq \widetilde S_1$. Consequently, no such $O$ can exist. Therefore, we need a more concrete method to construct examples with multiple NE without requiring the knowledge of the set $\widetilde S_\infty \setminus S_\infty$ as discussed above. Such a construction is provided in the following proposition.
	
        \begin{proposition} \label{prop: constructive}
            Let $V$ be the value function for the pair $(\psi,\phi)$ and let $I$ be a subset of $ \{ \psi < V < \phi \}$, where we assume that $ \{ \psi < V < \phi \} \neq \emptyset$. We define the function $\phi_c$ by 
                \begin{align*}
                    \forall x \in I, \quad \phi_c(x) \coloneqq  V(x)  \qquad \text{and} \qquad \forall x \notin I, \quad \phi_c(x) \coloneqq  \phi(x),
                \end{align*}
            and let $V_c$ be the value function for the corresponding pair $(\psi, \phi_c)$ under the new expected payoff $R^c_x(\cdot,\cdot)$ (with $\phi$ replaced by $\phi_c$ in \eqref{expected payoff}). Then it holds that $V_c = V$.
        \end{proposition}
        \begin{proof}
            By the necessity part of Theorem \ref{theo: sufficient for NE} in Corollary \ref{cor: necessity part}, $V$ satisfies the conditions \eqref{conditions for a NE}, but $V \leq \phi_c$ by construction so  the last condition in \eqref{conditions for a NE} can be replaced by $\psi \leq V \leq \phi_c$. But this implies $V = V_c$ again by Theorem \ref{theo: sufficient for NE}.
        \end{proof}
        
        Consider the case $\phi > \psi$. Applying Theorem \ref{theo: sufficient for NE}, we can easily verify that both pairs $h(\{V =\psi \}), h( \{V= \phi \}))$ and $h(\{V =\psi \}), h( \{V= \phi \} \cup I))$ are both NE for the game with functions $(\psi,\phi_c)$. Also note that $\{V= \phi \} \cup I = \{V= \phi_c \}$.
        
           \subsubsection{Computational aspects.}
	We discuss the computational aspects of our algorithm. By closely examining each iteration involved in computing a function $V_k$, we always start with the set $D^{(1)}_k = \{ \Q[\psi] - \beta \psi \leq 0 \} \setminus S_k$ and iteratively shrink this set until no further reduction is possible (potentially resulting in an empty set). The number of steps required to compute $V_k$ is thus bounded above by the cardinality of the set $\{\Q[\psi] - \beta \psi \leq 0\}$. To clarify, by "iteration," we refer to the index $k$ in the sequence $(V_k)_{k\geq 0}$, while "steps" refer to each update in the sequence $V_{k}^{(n)}$ for each $n=1,2,\dots$.
        
         Additionally, the sequence $(S_k)_{k \geq 1}$ decreases from $S_1 = \{ V_0 > \phi \} \cup \{ \phi = \psi \}$. Notably, by definition, the set $\{ \phi = \psi \}$ is always included in each member of this sequence. Therefore, we are "only" shrinking the set $\{ V_0 > \phi \}$, which is disjoint from $\{ \phi = \psi \}$. If we used $\widetilde S_1 = \{ V_0 \geq \phi \}$ instead, the set we shrink would be $ \widetilde S_1\setminus \{ \phi = \psi \} = \{ V \geq \phi > \psi \}$. In total, if  $\{V_0 > \phi \} \neq \emptyset$, the total number of steps, summed over all iterations required to compute $V$ is bounded above by 
            \begin{align*}
                |\{\Q[\psi] - \beta \psi \leq 0 \} | \times \big(| \{ V_0 > \phi \} | +1 \big) \leq |E|^2,
            \end{align*}
        where $|\cdot|$ denotes the cardinality and $| \{ V_0 > \phi \} | +1$ is an upper bound for the number of iterations. The additional plus 1 in the above inequality arises from the need for at least  $|\{\Q[\psi] - \beta \psi \leq 0 \} |$ iterations to compute $V_0$ and then check if $\{V_0 > \phi \} \neq \emptyset$.
            For example, consider the case where $\phi$ is strictly greater than $\psi$, i.e. $\{ \phi = \psi \}  = \emptyset$ and $\{ V_0 > \phi \} = \{ x_0 \}$ for some $x_0 \in E$. Then the function $V_1$ satisfies $\psi \leq V_1 \leq \phi$ and
                \begin{align*}
                    \forall x \in D_1, \quad \Q[V_1](x) - \beta V_1(x) \leq 0, \qquad  \forall x \notin D_1 \cup \{ x_0 \},\quad  \Q[V_1](x) - \beta V_1(x) = 0.
                \end{align*}
            We claim the $V_1$ is the value function, i.e. $V_1 = V$. To see this, observe that $\Q[V_1](x_0) - \beta V_1(x_0) > 0$ must hold. If this were not the case, then $V_1$ would be $\beta$-excessive and dominate $\psi$, which would in turn imply that $V_0 \leq V_1 \leq V_0$ (recall that $V_0$ is the smallest $\beta$-excessive function dominating $\psi$), leading to the conclusion $V_0 = V_1$. This is impossible because $V_0(x_0) > \phi(x_0) = V_1(x_0)$. By applying Theorem \ref{theo: sufficient for NE}, we conclude that $V = V_1$. Therefore, in this case, less than $2 \times | \{ \Q[\psi] - \beta \psi \leq 0 \} | \leq 2 |E|$ steps are required to compute $V_1 = V$.

        Finally, the following result demonstrates that if the players choose the NE $(h(\{V = \psi\}), h(\{V = \phi\}))$ as their strategies, then the game always stops after a finite period of time by the sup-player, if not by the inf-player. 
            \begin{theorem}
                For all $x \in E$, $\mathbb P_x( h(\{V = \psi\}) < + \infty ) = 1$.
            \end{theorem}
            \begin{proof}
                Since $ \psi \leq V \leq V_0$ from Theorem \ref{theo: big seq} and \ref{theo: main}, we have $\{V_0 = \psi \} \subset \{ V = \psi \}$. By Theorem \ref{ends after finite time}, we must have for all $x \in E$, $\mathbb P_x( h(\{V = \psi\}) < + \infty )  = 1$.
            \end{proof}
            
            
\section{Examples} \label{sec:examples}
	All of the examples presented in this section, along with the corresponding Python code, are available at the following link: https://github.com/nhatthangle/Markov-game.
    \subsection{Example 1: A birth-death process on $\{0,...,N-1\}$ with reflecting endpoints}
        In this example, we consider a birth-death process on the state space $\{0,1,...,N-1 \}$ with reflecting endpoints. The generator is given by 
            \begin{align*}
               \forall i \in \{ 1,...,N-2\}, \quad  &Q(i,i+1) = Q(0,1) =  \lambda, \quad Q(i,i-1) = Q(N-1,N-2) = r
            \end{align*}
        \subsubsection{Subexample 1.1.}
        We first consider the first sub-example, where the two function $\psi$ and $\phi$ are given by 
            \begin{align*}
                \forall x \in \{0,...,N-1 \}, \quad  \psi(x) = 10 + x/4 + 3 \cos(x) + 2 \sin (x/2), \qquad \text{and} \quad \phi = \psi +3 
            \end{align*}
        For the choice of 
        	\[N = 50, \quad \beta = 0.1, \quad \lambda = 40, \quad r = 28, \]
we have the following evolution of functions given in Figure \ref{fig: testfuncs}, where the algorithm stops after calculating $V_3 = V$. Note that all the functions $(V_k)_{k \geq 0}$ lie between $\psi$ and $\phi$, except $V_0$.
        \\
           The evolution of $(D_k)_{k = 1,2,3}$ and $(S_k)_{k = 1 ,2 ,3}$ are given by 
            \begin{align*}
                D_1 &= \{31, 38, 44, 49\},\\
                D_2 &= \{6, 7, 13, 18, 19, 25, 26, 31, 32, 38, 44, 49\}, \\
                D_3 &= \{0, 6, 7, 13, 18, 19, 25, 26, 31, 32, 38, 44, 49\}, \\
                S_1 &= \{0, 1, 2, 3, 4, 5, 7, 8, 9, 10, 11, 12, 14, 15, 16, 17, 20, 21,\\
                	& \hspace{3cm}  22, 23, 24, 27, 28, 29, 33, 34, 35, 36, 40, 41, 46, 47, 48\}, \\
                S_2 &=  \{2, 3, 4, 9, 10, 11, 15, 16, 21, 22, 23, 28, 34, 35, 41, 47\}, \\
                S_3 &=  \{3, 9, 10, 16, 22, 23, 28, 34, 35, 41, 47\}.
            \end{align*}
            \begin{figure}[t]
 		 \centering
 		 \includegraphics[width=\linewidth]{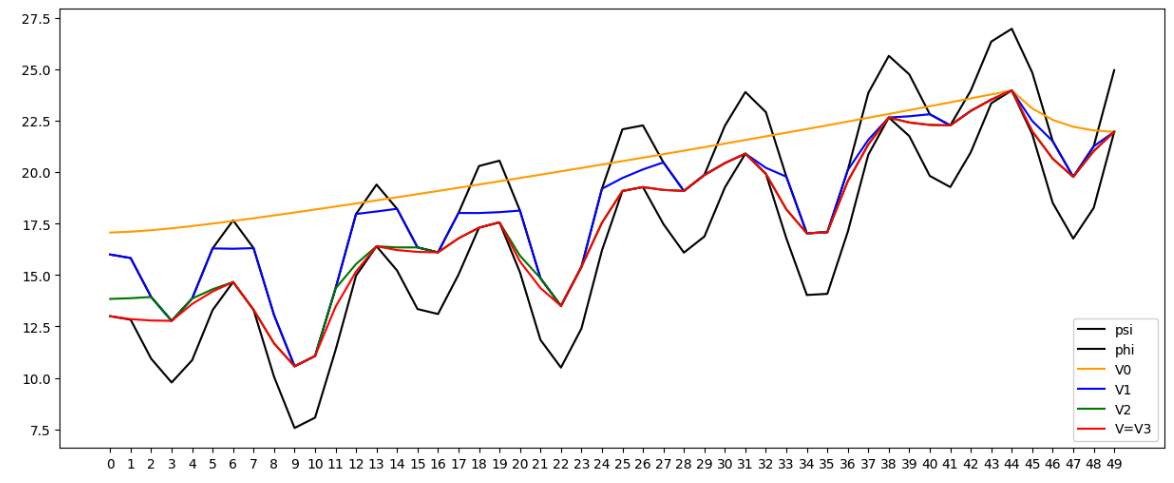}
  		  \caption{The evolution of $(V_k)_{k\geq 0}$ (here $V_3 = V$)} 
                 \label{fig: testfuncs}
	\end{figure}

    \subsubsection{Subexample 1.2.}
        In this subexample, we keep all $N, \beta, \lambda$ and $r$ as above and modify the function $\phi$ a little bit so that $\{ \phi = \psi \} \neq \emptyset$: 
            \begin{align*}
                \forall x \in \{0,...,49\}, \quad \phi(x) = \psi(x) + 4(\sin(x/5)+0.7)_+
            \end{align*}
        Then the evolution of functions $(V_k)_{k\geq 0}$ are given in Figure \ref{fig: testfuncs1}. Note that the sequence $(V_k)_{k\geq 0}$ still lies in the "sandwich" $\psi \leq \phi$, except $V_0$.
        The evolution of $(D_k)_{k = 1,2,3,4}$ and $(S_k)_{k = 1 ,2 ,3,4}$ in this case are given by 
            \begin{align*}
                D_1 &=  \{20, 21, 22, 23, 24, 27, 31, 38, 44, 49\}, \\
                D_2 &= \{13, 18, 19, 20, 21, 22, 23, 24, 27, 31, 32, 38, 44, 49\} ,  \\
                D_3 &= D_4 = \{6, 13, 18, 19, 20, 21, 22, 23, 24, 27, 31, 32, 38, 44, 49\},\\
                S_1 &= \{0, 1, 2, 3, 4, 5, 8, 9, 10, 11, 14, 15, 16, 17, 18, 19, 20, 21, 22, 23,\\
                & \hspace{3cm}  24, 25, 26, 27, 28, 29, 30, 33, 34, 35, 46, 47, 48\} \\ 
                S_2 &= \{0, 2, 3, 4, 9, 10, 15, 16, 17, 20, 21, 22, 23, 24, 25, 26, 27, 28, 29, 34, 35, 47, 48\},  \\
                S_3 &=  \{3, 9, 10, 16, 20, 21, 22, 23, 24, 25, 26, 27, 28, 29, 34, 35, 47, 48\}, \\
                S_4 &=  \{9, 10, 16, 20, 21, 22, 23, 24, 25, 26, 27, 28, 29, 34, 35, 47, 48\}.
            \end{align*}
            \begin{figure}
                 \centering
                 \includegraphics[width=\linewidth]{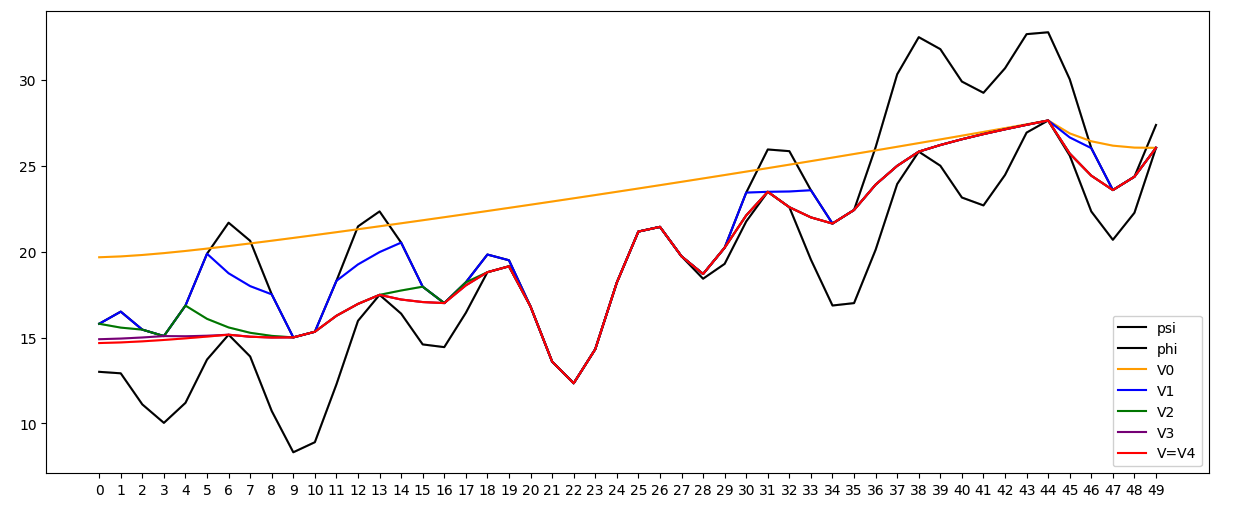}
                 \caption{The evolution of $(V_k)_{k\geq 0}$ (here $V_4=V$)} 
                 \label{fig: testfuncs1}
            \end{figure}
    \subsubsection{Subexample 1.3.}
        In this subexample, we choose $N = 50, \beta = 0.05, \lambda = 14, r = 12$. The functions $\psi$ and $\phi$ are defined by
            \begin{align*}
                \psi(x) = (-25+x)_+, \qquad \phi(x) = \psi(x) +5, \quad \forall x \in \{ 0,1,...,N-1\}
            \end{align*}
        The evolutions of $(D_k)_{k = 1,2,3,4,5,6}$ and $(S_k)_{k = 1 ,2 ,3,4,5,6}$ in this case are given by 
            \begin{align*}
                D_k =\{ 49\}, \ k= 1,2,3,4,5,6,  \ S_1 = \{0,1,2,...,39\}, \ S_k = \{ 25,26,..., 40 - k\}, \ k = 2,3,4,5,6.
            \end{align*}
        \begin{figure}
                 \centering
                 \includegraphics[width=\linewidth]{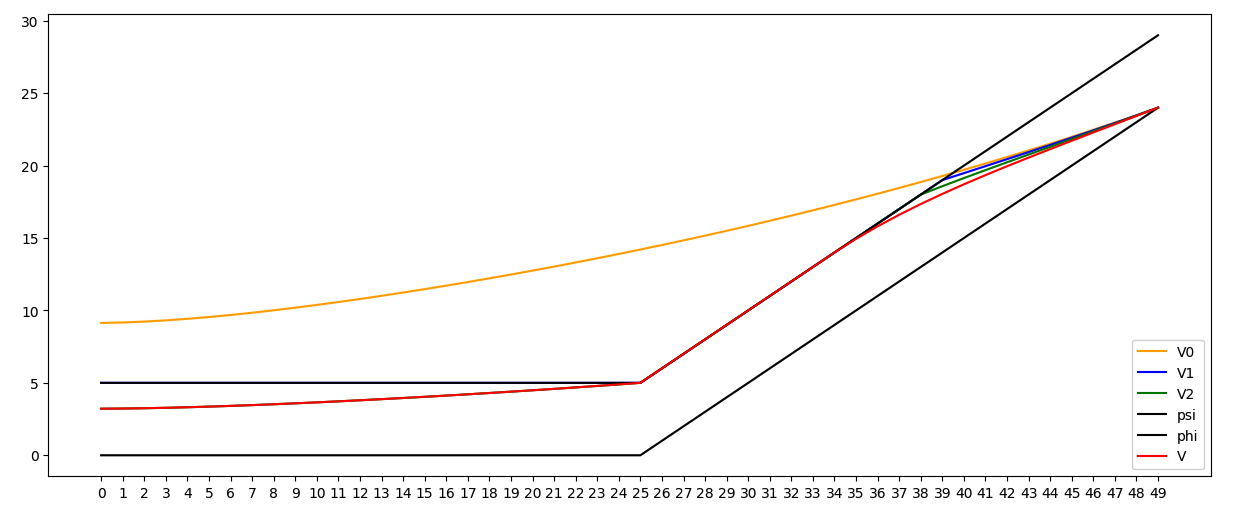}
                 \caption{The evolution of $(V_k)_{k\geq 0}$ (here $V_6 = V$)} .
                 \label{fig: finance}
            \end{figure}
        
    \subsection{Example 2: Random walk on a lattice}
    In this example, we consider a random walk on a lattice $\mathcal L$ with $N\times N$ nodes. This is a grid of points $(x,y)$, where $x, y \in \{0,1,...,N-1 \}$. Given $r >0$, we define the generator $\mathcal Q$ by,  for any  $1 \leq i,j \leq N-1$,
        \begin{align*}
             \quad &Q((i,j), (i+1,j)) = Q((i,j), (i-1,j)) = Q((i,j), (i,j+1))= Q((i,j), (i,j-1)) = r, 
        \end{align*}
 for all $ 1 \leq j \leq N-1$, 
        \begin{align*}
              Q((j,0),(j+1,0)) = Q((j,0),(j-1,0)) = Q((j,1),(j,0)) = r,\\
             Q((j,N-1),(j+1,N-1)) = Q((j,N-1),(j-1,N-1)) = Q((j,N-1),(j,N-2)) = r,
        \end{align*}
    and the rest of transitions are $0$. As can be readily seen, the lines $x = 0$ and $x = N-1$ are absorbing, i.e. when the process hits the lines $x = 0$ or $x = N-1$, it stops there forever. To make it work on a programming language like Python, we introduce the following bijection between the grid points and the sets $\{0,1,..., N^2-1\}$ as follows 
        \begin{align*}
            \forall (i,j) \in \mathcal L, \quad p((i,j)) \coloneqq  i +jN.
        \end{align*}
    In doing so, we regard $X = (X_t)_{t\geq 0}$ as a process on the set $\{ 0,..., N^2-1 \}$.
\subsubsection{Subexample 2.1.}
    Define the functions $\psi$ and $\phi$ as follows 
        \begin{align*}
            \psi(x)  = (x - N^2/2 )_+, \quad \phi(x) = \psi(x) + \delta, \quad \forall  x = 0,1,...,N^2-1,
        \end{align*}
    where $\delta > 0$ is a fixed constant. Note that the closer to the line $y = N-1$ the process is (corresponding to the values near $N^2-1$), the bigger $\psi$ and $\phi$ gets.\\ 

    In what follows, we choose $N = 13$, $\beta = 0.05$, $r = 5$ and $\delta = 8$. The algorithm gives  the following evolution of the sets $(D_k)_{k\geq 1}$ (points in blue) and $(S_k)_{k \geq 1}$ (in red) and stops after calculating $V_4$.
\begin{figure}
    \centering
    \begin{subfigure}{0.24\linewidth}
        \centering
        \includegraphics[width=\linewidth]{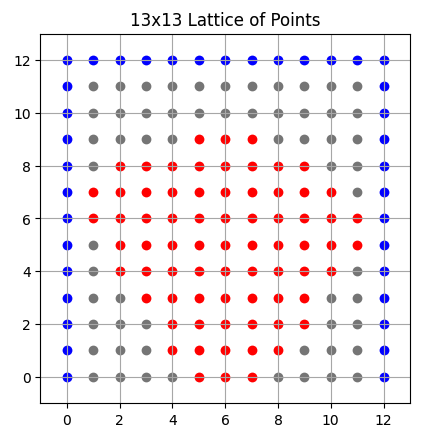}
    \end{subfigure}
    \hfill
    \begin{subfigure}{0.24\linewidth}
        \centering
        \includegraphics[width=\linewidth]{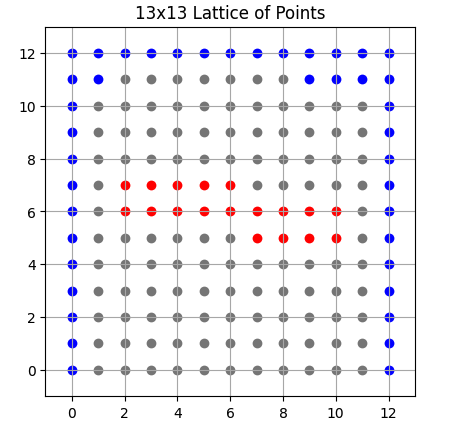}
    \end{subfigure}
    \hfill
    \begin{subfigure}{0.24\linewidth}
        \centering
        \includegraphics[width=\linewidth]{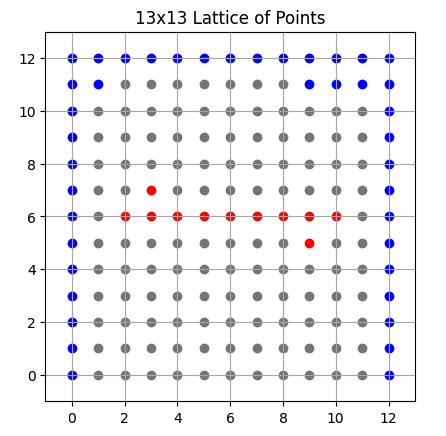}
    \end{subfigure}
    \hfill
    \begin{subfigure}{0.24\linewidth}
        \centering
        \includegraphics[width=\linewidth]{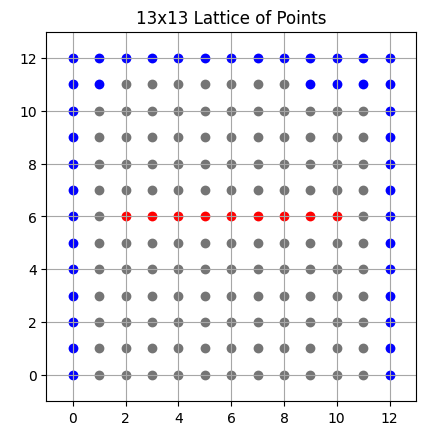}
    \end{subfigure}

    \caption{The evolution of stopping regions $(D_k,S_k)_{k= 1,2,3,4}$ from left to right.}
    \label{fig:DSk}
\end{figure}
    In view of functions on $\{0,1,...,N^2-1\}$, we can draw $\psi, \phi$ (in connected lines), $V_0$ and the value function $V$ in 1-dimension in Figure \ref{fig: functions1} below. 
        \begin{figure}
                 \centering
                 \includegraphics[width=\linewidth]{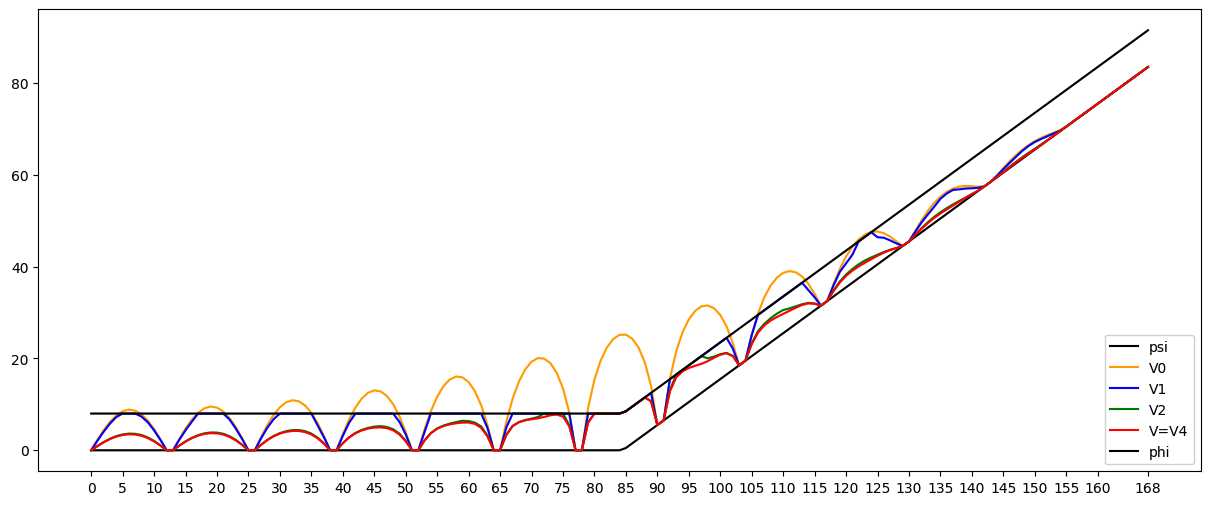}
                 \caption{The functions $\psi, \phi, V_0, V$ on $\{0,1,...,168\}$ ($N = 13$).} 
                 \label{fig: functions1}
        \end{figure}
\subsubsection{Subexample 2.2.}
    We now consider a second sub-example in which the function $\phi$ is slightly modified so that the set $\{\phi = \psi\} \neq \emptyset.$ Specifically, we set $\phi = 1.5\,\psi$ and increase the rate parameter to $r = 500$ and $\beta = 1$. The evolution of the sets $(D_k, S_k)_{k \geq 1}$ is illustrated in Figure~\ref{fig: DSk2}. The set $\{\phi = \psi\}$ is denoted by green points.  The elements of $S_k \setminus \{\phi = \psi\}$ are represented by red points, so that $S_k$ is the union of the  triangles and squares. Finally, the sets $D_k$ are depicted by blue points.

The algorithm terminates after computing $V_4$, which corresponds to the final value function.
        \begin{figure}
                 \centering
    \begin{subfigure}{0.24\linewidth}
        \centering
        \includegraphics[width=\linewidth]{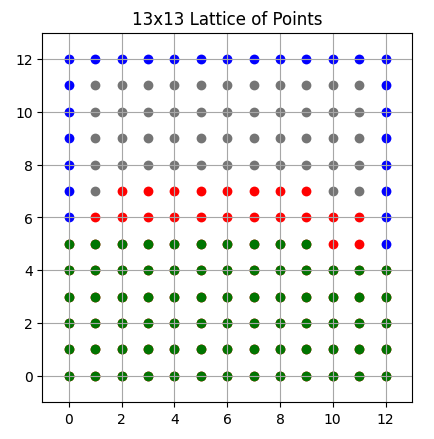}
    \end{subfigure}
    \hfill
    \begin{subfigure}{0.24\linewidth}
        \centering
        \includegraphics[width=\linewidth]{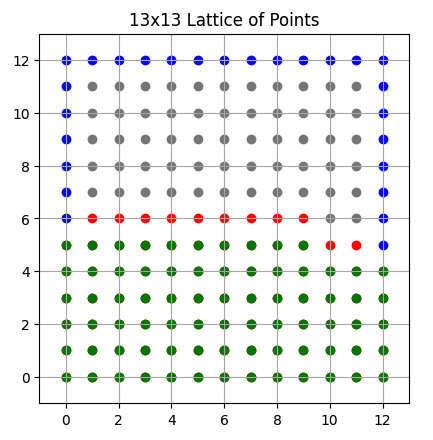}
    \end{subfigure}
    \hfill
    \begin{subfigure}{0.24\linewidth}
        \centering
        \includegraphics[width=\linewidth]{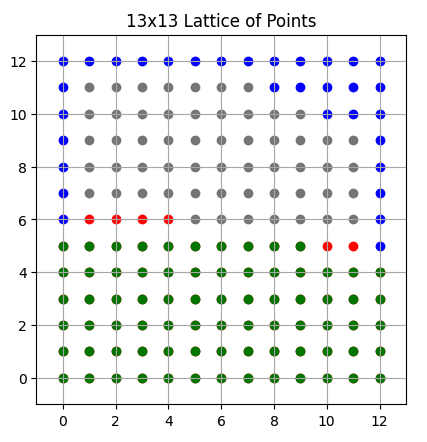}
    \end{subfigure}
    \hfill
    \begin{subfigure}{0.24\linewidth}
        \centering
        \includegraphics[width=\linewidth]{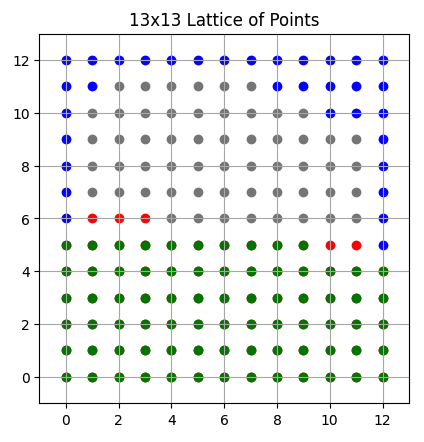}
    \end{subfigure}
                  \caption{The evolution of $(D_k,S_k)_{k\geq 1}$} 
                 \label{fig: DSk2}
        \end{figure}
    In one dimension, we can view the evolution of $(V_k)_{k \geq 0}$ as in Figure \ref{fig: functions3}. The function $V_3$ is omitted from the figure because its graph overlaps significantly with other curves, making it difficult to distinguish visually.
        \begin{figure}
                 \centering
                 \includegraphics[width=\linewidth]{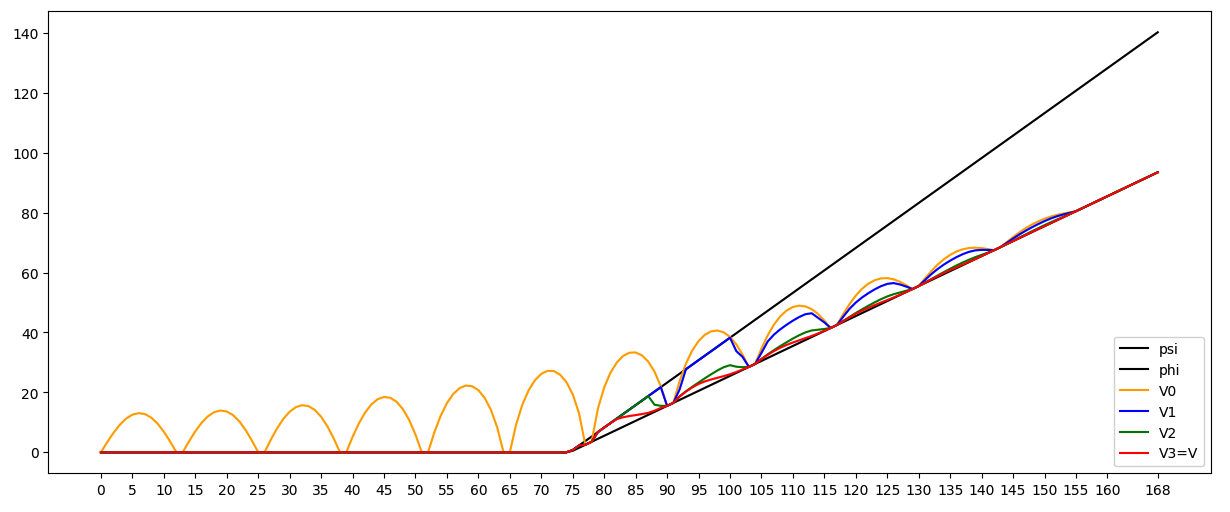}
                 \caption{The functions $\psi, \phi, V_0, V_1, V_2,  V_4 = V$ on $\{0,1,...,168\}$ ($N = 13$)} 
                 \label{fig: functions3}
    \end{figure}
    
    \subsection{Examples when $\widetilde S_\infty \neq S_\infty$ and $ \widetilde S_\infty = S_\infty $} \label{differentlimiting}
    In this section we consider $E = \{0,1,2,3 \}$, $\beta = 0.2$ and the generator $\Q$ is given by 
		\begin{align*}
		 \forall i =0,1,2, \quad Q(i,i+1) =1, \quad \text{and} \quad \forall i = 1,2,3, \quad   Q(i,i-1) =1.
		\end{align*}
	\subsubsection{The case when $\widetilde S_\infty = S_\infty$.}
	 In this subsection, $\psi$ is chosen such that
    	\begin{align*}
		\psi(0) =10, \ \psi(1) =4, \ \psi(2) =2, \ \psi(3) =1.
	\end{align*}
	We then compute $V_0$ and choose $\phi$ as follow
	\begin{align*}
		V_0(0) &= 10, \quad V_0(1) =  6.8106 , \quad V_0(2) =  4.9834, \quad V_0(3) = 4.1528, \\
		\phi(0) &=12, \quad \phi(1) =8, \quad \phi(2) = V_0(2), \quad \phi(3) =1.
	\end{align*}
	In this case, $S_1 = \{ 3\}$ and $\widetilde S_1 = \{2,3\}$ and the limiting sets are equal $\widetilde S_\infty = S_\infty = \{  3 \}$. This is illustrated in Figure \ref{fig: equal}.
                 \begin{figure}
                 \centering
    \begin{subfigure}{0.48\linewidth}
        \centering
        \includegraphics[width=\linewidth]{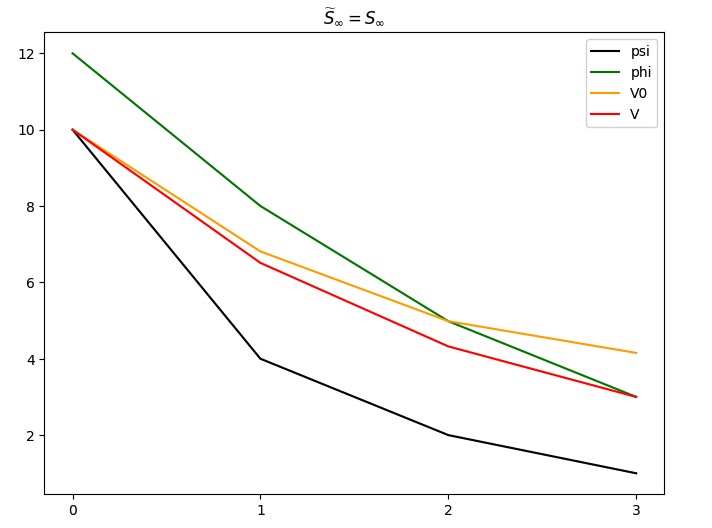}
        	\caption{The case when $\widetilde S_\infty = S_\infty$} 
	\label{fig: equal}
    \end{subfigure}
    \hfill
    \begin{subfigure}{0.48\linewidth}
        \centering
        \includegraphics[width=\linewidth]{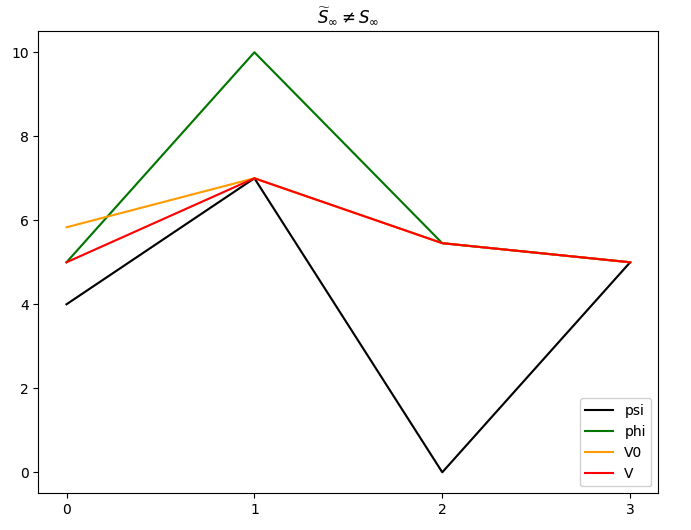}
        \caption{The case when $\widetilde S_\infty \neq S_\infty$} 
            \label{fig: neq}
    \end{subfigure}
    \hfill
    \caption{Different initializations may lead to different NE.}
    \label{fig:diff}
    \end{figure}
    \subsubsection{The case when $\widetilde S_\infty \neq S_\infty$.}
	 In this subsection, $\psi$ is chosen such that
    	\begin{align*}
		\psi(0) = 4, \ \psi(1) = 7, \ \psi(2) =0, \ \psi(3) =5.
	\end{align*}
	We then compute $V_0$ and choose $\phi$ as follow
	\begin{align*}
		V_0(0) &=  35/6, \quad V_0(1) =  7 , \quad V_0(2) = 60/11, \quad V_0(3) = 5, \\
		\phi(0) &= 5, \quad \phi(1) = 10, \quad \phi(2) = V_0(2) = 60/11, \quad \phi(3) = V_0(3) = 5.
	\end{align*}
	In this case, $S_1 = \{ 0, 3\}$ and $\widetilde S_1 = \{0, 2,3\}$ and the limiting sets are $S_\infty = \{  0, 3 \}$ while $\widetilde S_\infty= \{ 0,2,3\}$. This example is illustrated in Figure \ref{fig: neq}.

\end{document}